\documentclass[11pt]{article}

\usepackage{amssymb}
\usepackage{amsthm}
\usepackage{amsmath}
\usepackage{mathrsfs}

\usepackage{color}
\usepackage{graphicx,subcaption}

\usepackage{indentfirst}
\usepackage[shortlabels]{enumitem}
\usepackage[titletoc]{appendix}
\usepackage{hyperref}
\hypersetup{colorlinks=false, citecolor=green, linkcolor=red, hypertexnames=false}
\usepackage{titlesec}
\titleformat{\section}
  {\vskip 1em\normalfont\Large\sc\centering}{\thesection}{1em}{}
\usepackage[justification=centering]{caption}

\newtheorem{Theorem}{Theorem}[section]
\newtheorem*{Main Theorem}{Main Theorem}
\newtheorem*{Theorem A}{Theorem A}
\newtheorem*{Theorem B}{Theorem B}
\newtheorem*{Theorem C}{Theorem C}
\newtheorem*{Theorem D}{Theorem D}

\newtheorem{Definition}[Theorem]{Definition}
\newtheorem{Proposition}[Theorem]{Proposition}
\newtheorem{Lemma}[Theorem]{Lemma}
\newtheorem{Question}{Question}

\newtheorem{Corollary}[Theorem]{Corollary}

\theoremstyle{remark}
\newtheorem{Remark}{Remark}[section]

 \def\RR{{\mathbb R}}

\def\si{\sigma}
\def\La{\Lambda}
\def\la{\lambda}

\def\Cl{{\rm Cl}}

\def\cB{\mathcal{B}}  \def\cH{\mathcal{H}} \def\cN{\mathcal{N}} \def\cT{\mathcal{T}}
\def\cC{\mathcal{C}}    \def\cU{\mathcal{U}}
\def\cD{\mathcal{D}}   \def\cP{\mathcal{P}} \def\cV{\mathcal{V}}
\def\cE{\mathcal{E}}    \def\cW{\mathcal{W}}
\def\cF{\mathcal{F}}

\def\xX{\mathscr{X}}

\def\dim{\operatorname{dim}}
\def\ind{\operatorname{Ind}}

\def\Sing{\operatorname{Sing}}
\def\orb{\operatorname{Orb}}

\def\vep{{\varepsilon}}

\def\e{\mathrm{e}}
\def\wt{\widetilde}
\def\len{\operatorname{len}}

\numberwithin{equation}{section}

\title{An example derived from Lorenz attractor}

\author{Ming Li, Fan Yang, Jiagang Yang, Rusong Zheng\thanks{M. Li is supported by NSF of China (11971246, 12071231) and National Key R$\&$D Program of China (2020YFA0713300); J. Yang is supported by NSF of China (12271538, 11871487, 12071202), CNPq of Brazil, FAPERJ of Brazil, and PRONEX of Brazil; R. Zheng is supported by NSF of China (12071007, 12101293).}}
\begin{document}
	
\maketitle

\begin{abstract}
We consider a DA-type surgery of the famous Lorenz attractor in dimension 4. This kind of surgeries have been firstly used by Smale \cite{Sm67} and Ma\~n\'e \cite{Ma} to give important examples in the study of partially hyperbolic systems. Our construction gives the first example of a singular chain recurrence class which is Lyapunov stable, away from homoclinic tangencies and exhibits robustly heterodimensional cycles. Moreover, the chain recurrence class has the following interesting property: there exists robustly a 2-dimensional sectionally expanding subbundle (containing the flow direction) of the tangent bundle such that it is properly included in a subbundle of the finest dominated splitting for the tangent flow.
\end{abstract}

\section{Introduction}

\subsection{Motivation}

In the study of differentiable dynamical systems, a central problem is to understand most of the systems. From a topological viewpoint, this could mean to characterize an open and dense subset, or at least a residual subset, of all systems.
It was once conjectured by Smale in the early 1960's that an open and dense subset of all systems should be comprised of uniformly hyperbolic ones. Recall that a dynamical system is {\em uniformly hyperbolic} if its limit set consists of finitely many transitive hyperblic sets and without cycles in between. The first counterexamples to this conjecture were given by Abraham and Smale \cite{AS} in the $C^1$ topology, and by Newhouse \cite{Ne,Ne2} in the $C^2$ topology. These examples give the following two obstructions to hyperbolicity:
\begin{enumerate}
  \item {\em heterodimensional cycle}\,: there exist two hyperbolic periodic orbits of different indices such that the stable manifold of one periodic orbit intersects the unstable manifold of the other and vice versa;
  \item {\em homoclinic tangency}\,: there exists a hyperbolic periodic orbit whose stable and unstable manifolds have a nontransverse intersection.
\end{enumerate}
Bifurcations of these two mechanisms can exhibit very rich dynamical phenomena. For example, a heterodimensional cycle associated to a pair of periodic saddles with index $i$ and $j=i+1$ can be perturbed to obtain a {\em robust heterodimensional cycle} \cite{BD}: there exist transitive hyperbolic sets $\La_1$ and $\La_2$ with different indices such that the stable manifold of $\La_1$ meets the unstable manifold of $\La_2$ in a $C^1$ robust way, and vice versa. For more results, see e.g. \cite{PT,D,BDPR}.
Trying to give a global view of all dynamical systems, Palis \cite{Pa90, Pa00, Pa05, Pa08} conjectured in the 1990’s that the uniformly hyperbolic systems and the systems with heterodimensional cycles or homoclinic tangencies form a dense subset of all diffeomorphisms.
The Palis conjecture has been proved by Pujals and Sambarino \cite{PS} for $C^1$ diffeomorphisms on surfaces. Moreover, a weaker version, known as the weak Palis conjecture, was shown to hold for $C^1$ diffeomorphisms on any closed manifold \cite{BGW, C}.
When it comes to flows, or vector fields, the Palis conjecture has to be recast so as to take into consideration the {\em singularities}.
There have been many progresses on this topic, see e.g. \cite{ARH,GY,CY18}. Also, related conjectures have been given by Bonatti \cite{B} in his program for a global view of $C^1$ systems.

The conjectures by Palis and Bonatti have lead to the study of systems away from homoclinic tangencies and/or heterodimensional cycles. In the $C^1$ topology, a diffeomorphism (or a vector field) is said to be {\em away from homoclinic tangencies} if every system in a $C^1$ neighborhood exhibits no homoclinic tangency. It is now well understood \cite{W02,Go} that for a diffeomorphism away from homoclinic tangencies there exist dominated splittings of the tangent bundle over preperiodic sets. More delicate results can be obtained in the $C^1$ generic setting, see e.g. \cite{CSY}. Also, it is known \cite{LVY} that every $C^1$ diffeomorphism away from homoclinic tangencies is entropy expansive, and it follows that there exists a measure of maximum entropy.
When considering diffeomorphisms away from homoclinic tangencies {\em and} heterodimensional cycles, it is shown in \cite{CP} that a $C^1$ generic system is essentially hyperbolic. See also \cite{C2}.

While there have been plenty of results on diffeomorphisms away from homoclinic tangencies (and heterodimensioinal cycles), parallel studies on vector fields are rare.
Regarding the Palis conjecture, it is shown \cite{CY18} that on every 3-manifold the singular hyperbolic vector fields form a $C^1$ open and dense subset of all vector fields away from homoclinic tangencies. Also in dimension 3, the weak Palis conjecture has been proved for $C^1$ vector fields \cite{GY}. For higher dimensional (dim $\geq 3$) vector fields, studies are mostly limited to either the singular hyperbolic vector fields \cite{MPP} or the star vector fields, see e.g. \cite{SGW,PYY21,CY,dL,BdL}. General vector fields away from homoclinic tangencies (not star) on a higher dimensional manifold are rarely studied. In \cite{GYZ}, for general higher dimensional flows away from homoclinic tangencies, it is shown that $C^1$ generically a nontrivial Lyapunov stable chain recurrence class contains periodic orbits and hence is a homoclinic class. Their result leaves the possibility that the periodic orbits in the chain recurrence class can have different indices. When this happens robustly, such a chain recurrence class shall exhibit a robust heterodimensional cycle \cite{BD} and the flow can not be star.
Apart from this result and possibly a handful others, the world of general higher dimensional vector fields away from homoclinic tangencies is far from being understood.

Inspired by the result of \cite{GYZ}, we notice that there is no known example of a singular flow which is away from homoclinic tangencies and is not singular hyperbolic or star.
The aim of this paper is to provide such an example, and to open a door to the world of general higher dimensional vector fields away from homoclinic tangencies.
Hopefully it will also attract interests to the study of higher dimensional vector fields.

As a by-product, our example exhibits an interesting property which is not observed before: the tangent bundle of the obtained chain recurrence class is shown to admit robustly an invariant subbundle (containing the flow direction), which is {\em properly} included in a subbundle of the finest dominated spitting of the tangent bundle. We shall call such an invariant subbundle {\em exceptional} (see Definition \ref{def.exceptional-bundle}).
Note that there can not be any exceptional subbundle for diffeomorphisms away from homoclinic tangencies: robustness of an invariant subbundle implies a corresponding dominated spitting of the tangent bundle. Thus the existence of exceptional subbundles may be counted as an essential difference between diffeomorphisms and (singular) vector fields, and is worth studying.

\subsection{Precise statement of our result}

Let $M$ be a compact Riemannian manifold without boundary.
Given a $C^1$ vector field on $M$, it generates a $C^1$ flow $\phi^X_t$, or simply $\phi_t$, on the manifold. A point $x\in M$ with $X(x)=0$ is called a {\em singularity} of $X$. Any point other than singularities is called a {\em regular} point. A regular point $x$ satisfying $\phi_T(x)=x$ for some $T>0$ is called a {\em periodic} point, and its orbit $\orb(x)=\{\phi_t(x): t\in\RR\}$ will be called a {\em periodic orbit}. Singularities and periodic orbits are called {\em critical elements}.

Let $\La$ be a compact invariant set of the flow $\phi_t$. It is called {\em chain transitive} if for any two points $x,y\in\Lambda$ and any constant $\vep>0$, there exist a sequence of points $x_0=x,x_1,\ldots,x_n=y$ in $\Lambda$ and a sequence of positive numbers $t_0, t_1,\ldots,t_{n-1}\geq 1$, such that $d(\phi_{t_i}(x_i),x_{i+1})<\vep$ for all $i=0,1,\ldots,n-1$. If $\Lambda$ is {\em maximally chain transitive}, i.e. it is chain transitive and is not a proper subset of any other chain transitive set, then $\Lambda$ is called a {\em chain recurrence class}. The union of all chain recurrence classes is called the chain recurrent set, denoted by $\mathrm{CR}(X)$. For any point $x$ in $\mathrm{CR}(X)$, we denote by $C(x,X)$ the chain recurrence class containing $x$.

Denote by $\Phi_t$ the tangent flow of $X$, i.e. $\Phi_t=D\phi_t$.
A $\Phi_t$-invariant subbundle $E$ of $T_{\La}M$ with $\dim E\geq 2$ is called {\em sectionally expanded} if there exists constants $C\geq 1$, $\lambda>0$ such that for any $x\in \La$ and any 2-dimensional subspace $S\subset E(x)$, it holds
\[|\det(\Phi_{-t}|_S)|\leq C\e^{-\la t}, \quad \forall t\geq 0.\]
A $\Phi_t$-invariant splitting $E\oplus F\subset T_{\La}M$ is called {\em dominated} if there exists $C\geq 1$, $\lambda>0$ such that for any $x\in \La$, it holds
\[\|\Phi_t|_{E(x)}\|\cdot\|\Phi_{-t}|_{F(\phi_t(x))}\|\leq C\e^{-\la t},\quad \forall t\geq 0.\]
We also say that a $\Phi_t$-invariant splitting $E_1\oplus E_2\oplus\cdots\oplus E_k\subset T_{\La}M$ is dominated if $(E_1\oplus\cdots\oplus E_i)\oplus (E_{i+1}\oplus\cdots\oplus E_k)$ is dominated for every $i=1,\ldots,k-1$.
The {\em finest dominated splitting} of $T_{\La}M$ is a $\Phi_t$-invariant splitting $T_{\La}M=E_1\oplus E_2\oplus \cdots\oplus E_k$ such that it is dominated and no subbundle $E_i$ ($i=1,\ldots,k$) can be decomposed further into a dominated splitting.

\begin{Definition}\label{def.exceptional-bundle}
  Let $\La$ be a chain transitive compact invariant set of a $C^1$ vector field $X$, a subbundle $E\subset T_{\La}M$ is called {\em exceptional} if it satisfies all the following properties:
\begin{itemize}
  \item $\dim E\geq 2$, and $X(x)\in E(x)$ for any $x\in\La$;
  \item $E$ is invariant for the tangent flow;
  \item $E$ is properly included in a subbundle of the finest dominated splitting of $T_{\La}M$.
\end{itemize}
\end{Definition}
In particular, if $E\subset T_{\La}M$ is exceptional, it is not a subbundle of any dominated splitting of $T_{\La}M$.
Recall that a compact invariant set $\La$ is called {\em Lyapunov stable} if for any neighborhood $U$ of $\La$ there exists a neighborhood $V$ of $\La$ such that $\overline{\phi_t(V)}\subset U$ for all $t\geq 0$.
Let $\xX^1(M)$ be the set of all $C^1$ vector fields on $M$, endowed with the $C^1$ topology. We denote by $\cH\cT$ the set of all $C^1$ vector fields on $M$ exhibiting a homoclinic tangency. Then $\xX^1(M)\setminus\overline{\cH\cT}$ consists of all $C^1$ vector fields away from homoclinic tangencies. %We say that a chain recurrence class {\em contains a robust heterodimensional cycle} if it contains transitive hyperbolic sets $\La_1$ and $\La_2$ with different indices that exhibit a robust heterodimensional cycle.
\begin{Theorem A}
On every 4-dimensional closed Riemannian manifold $M$ there exists a non-empty open set $\cV\subset \xX^1(M)\setminus\overline{\cH\cT}$ and an open set $U\subset M$ such that for any $X\in\cV$, one has $\phi^X_t(\overline{U})\subset U$ for any $t>0$, and
\begin{itemize}
  \item there is a unique singularity $\si_X\in U$ and it is hyperbolic;
  \item $C(\si_X,X)$ is the unique Lyapunov stable chain recurrence class contained in $U$;
  \item $C(\si_X,X)$ contains a robust heterodimensional cycle;
  \item the tangent bundle $T_{C(\si_X,X)}M$ contains a 2-dimensional subbundle $E^X$ which is exceptional and sectionally expanded, and it varies upper semi-continuously with respect to $X$.
\end{itemize}
\end{Theorem A}

\begin{Remark}
  Since a robust heterodimensional cycle yields non-hyperbolic periodic orbits by arbitrarily small $C^1$ perturbations,
  the chain recurrence class $C(\si_X,X)$ for any $X\in\cV$ is not singular hyperbolic or star.
\end{Remark}
\begin{Remark}
  In \cite{BLY}, a non-empty $C^1$ open set of vector fields having an attractor containing singularities of different indices is constructed on a given 4-manifold. Their example exhibits also a robust heterodimensional cycle, but is accumulated by vector fields with homoclinic tangencies.
\end{Remark}

\subsection{Idea of the construction}
The example (in Theorem A) is derived from Lorenz attractor in the sense that the construction begins with the Lorenz attractor and uses a DA-type surgery. The Lorenz attractor, discovered by the meteorologist E. Lorenz in the early 1960's \cite{Lo}, is among the most important examples of singular flows. It is a chaotic attractor given by a simple group of three ordinary differential equations. We will consider the geometric model of the Lorenz attractor \cite{ABS1, GW}, which is equivalent to the original attractor for classical parameters \cite{T}. The DA surgery was introduced by Smale \cite{Sm67} as a way of constructing nontrivial basic sets, see also \cite{Wi}. It has been used to give many inspiring and important examples, especially in the study of partially hyperbolic systems, see e.g. \cite{Ma,BV}.

In the following we give some ideas of the construction.

We begin with a vector field $X^0$ on the 3-ball $B^3$ having a Lorenz attractor $\La_0$. Then we consider a vector field $X(x,s)=(X^0(x),-\theta s)$ on the product $B^3\times [-1,1]$, where $(x,s)$ are coordinates on $B^3\times [-1,1]$ and $\theta>0$ is a constant. In particular, the dynamics is contracting along fibers. Note that $\La=\La_0\times\{0\}$ is an attractor and it contains a periodic orbit $P$.

Now the idea is to modify the vector field $X$ in a neighborhood of the periodic orbit $P$ so that one obtains a robust heterodimensional cycle. For this purpose we combine a DA-type surgery with the construction of a blender \cite{BDV}. Roughly speaking, a blender is a hyperbolic set which verifies some specific geometric properties so that it unstable manifold looks like a manifold of higher dimension. It naturally appears in the unfolding of a heterodimensional cycle of two periodic orbits and can be used to obtain robust heterodimensional cycles.

For this construction, we need to show that the obtained vector field is away from homoclinic tangencies. This requires a careful choice of the contraction rate $\theta$ along the fibers. Moreover, we need to show that the singularity and the heterodiemnsional cycle are contained in the same chain recurrence class. This constitutes the most difficult part of our proof.

\subsection{Further discussions}
For the example in Theorem A, we would like to raise the following questions.
\begin{Question}
  Is the chain recurrence class $C(\si_Y,Y)$ (robustly) transitive? Or even isolated?
\end{Question}

Apart from topological properties, one can also studies ergodic properties of the example. In \cite{SYY} it is proved that the entropy function for singular flows away from homoclinic tangencies is upper semi-continuous with respect to both invariant measures and the flows, provided that the limiting measure is not supported on singularities.
\begin{Question}
  Is the entropy function upper semi-continuous for the chain recurrence class $C(\si_Y,Y)$ in Theorem A?
\end{Question}
One may even consider existence of SRB measures or physical measures.
However, let us mention one difficulty in the study of this example. In contrast to the Lorenz attractor or a singular hyperbolic attractor, the usual analysis of the return map to a global cross-section may not be applicable in our example. This is because the chain recurrence class in our example admits a partially hyperbolic splitting $E^s\oplus F$, where the 3-dimensional center-unstable subbundle $F$ contains an exceptional 2-dimensional subbundle and is not sectionally expanded.

Apart from the present example, it may also be interesting to consider a DA-type surgery on the two-sided Lorenz attractor constructed by Barros, Bonatti and Pacifico \cite{BBP}.

The rest of this paper is organized as follows. We present some preliminaries of flows in Section \ref{sect.preliminaries} and give the construction of the example in Section \ref{sect.construction}. The proof of Theorem A is also given in Section \ref{sect.construction}, assuming some robust properties that will be proved in subsequent sections: in Section \ref{sect.away-from-ht}  we show that the example is away from homoclinic tangencies; then in Section \ref{sect.sectional-expanding} we show existence of an exceptional subbundle; finally in Section \ref{sect.ls} we show that the singularity and the heterdimensional cycle are contained in the same chain recurrence class.

\section{Preliminaries}\label{sect.preliminaries}
In this section we give some preliminaries on vector fields, including hyperboicity, linear Poincar\'e flow, fundamental limit, and so on.

Let $M$ be a compact Riemannian manifold without boundary. Let $X$ be a $C^1$ vector field on $M$.
Denoted by $\Sing(X)$ the set of singularities, i.e. $\Sing(X)=\{x\in M: X(x)=0\}$.
For any point $x\in M$, we denote by $\langle X(x)\rangle$ the subspace of $T_xM$ that is generated by $X(x)$. Note that $\langle X(x)\rangle$ is a one-dimensional subspace of $T_xM$ when $x\not\in \Sing(X)$.
To simplify notations, we also denote $l^X_x=\langle X(x)\rangle$ for $x\notin\Sing(X)$.

\subsection{Hyperbolicity}

Recall that the flow and the tangent flow generated by $X$ are denoted as $\phi_t$ and $\Phi_t$, respectively.
Let $\La$ be a compact $\phi_t$-invariant set and $E\subset T_{\La}M$ a $\Phi_t$-invariant subbundle. The dynamics $\Phi_t|_E$ is {\em contracting}, or $E$ is {\em contracted}, if there exists constants $C\geq 1$, $\lambda >0$, such that
\[\|\Phi_t|_{E(x)}\|\leq C\e^{-\lambda t},\quad \text{for any}\ x\in \Lambda, \ t\geq 0.\]
One says that $\Phi_t|_E$ is {\em expanding} if $\Phi_{-t}|_E$ is contracting. The set $\La$ is called {\em hyperbolic} if its tangent bundle admits a continuous splitting $T_{\La}M=E^s\oplus\langle X\rangle\oplus E^u$, where $\langle X\rangle$ is the subbundle generated by the vector field $X$, and $E^s$ (resp. $E^u$) is contracted (resp. expanded). When $\La$ is hyperbolic, for every $x\in\La$, the sets
\[W^{ss}(x)=\{y\in M: \mathrm{dist}(\phi_t(x), \phi_t(y))\to 0\ \text{as}\ t\to +\infty\}\]
and
\[W^{uu}(x)=\{y\in M: \mathrm{dist}(\phi_t(x), \phi_t(y))\to 0\ \text{as}\ t\to -\infty\}\]
are invariant $C^1$ submanifolds tangent to $E^s(x)$ and $E^u(y)$ respectively at $x$, see \cite{HPS}. We denote $W^s(\orb(x))=\bigcup_{z\in\orb(x)}W^{ss}(z)$ and $W^u(\orb(x))=\bigcup_{z\in\orb(x)}W^{uu}(z)$, which are the stable and unstable manifolds of $\orb(x)$ respectively. For a chain transitive hyperbolic set, all stable manifolds $W^{ss}(x)$ have the same dimension $\dim E^s$, which will be called the {\em stable index} of the hyperbolic set. Sometimes we simply say {\em index} with the same meaning of the stable index.

Flows with singularities, or singular flows, can exhibit complex dynamics with regular orbits accumulating on singularities, such as the famous Lorenz attractor \cite{Lo,GW}. In such cases, the set $\La$ can not be hyperbolic. Nevertheless, a weaker version of hyperbolicity can be defined.
One says that $\Lambda$ is {\em partially hyperbolic} if there exists a $\Phi_t$-invariant splitting of the tangent bundle $T_{\Lambda}M=E^s\oplus E^{cu}$ such that $E^s$ is contracted and it is dominated by $E^{cu}$.
If, moreover, the bundle $E^{cu}$ is sectionally expanded, then $\Lambda$ is said to be {\em singular hyperbolic} \cite{MPP}.
Note that the partial hyperbolicity or singular hyperbolicity ensures also the existence of stable manifolds.

We will study mainly chain recurrence classes.
A chain recurrence class $C(x,X)$ is {\em nontrivial} if it is not reduced to a critical element.
In \cite{PYY21}, it is shown that a nontrivial, Lyapunov stable, and singular hyperbolic chain recurrence class contains a periodic orbit, and $C^1$ generically, such a chain recurrence class is indeed an attractor. This result has been generalized to $C^1$ open and dense vector fields in \cite{CY}.

\begin{Lemma}[\cite{CY}]
	There is an open and dense set $\cU \in \xX^1(M)$ such that for any $X\in \cU$, any nontrivial, Lyapunov stable, and singular hyperbolic chain recurrence class $C(\sigma)$ of $X$ is a homoclinic class. Moreover, for any hyperbolic periodic orbit $\gamma\subset C(\sigma)$,  there exists a neighborhood $U$ of $C(\sigma)$ and a neighborhood $\cU_X$ of $X$ such that for any $Y\in\cU_X$, the stable manifold $W^s(\gamma_Y,Y)$ is dense in $U$.
\end{Lemma}

\subsection{Linear Poincar\'e flow and its extension}

Given a $\phi_t$-invariant set $\Lambda\subset M$. Define the normal bundle over $\Lambda$ as
\[\cN_{\Lambda}=\bigcup_{x\in\Lambda\setminus\Sing(X)}\cN_x,\]
where $\cN_x$ is the orthogonal complement of the one-dimensional subspace $l^X_x\subset T_xM$ generated by $X(x)$. In particular, if $\Lambda=M$, then $\cN_M=\bigcup\limits_{M\setminus \Sing(X)}\cN_x$.
For any $x\in M\setminus\Sing(X)$, $v\in \cN_x$ and $t\in\RR$, let $\psi_t(v)$ to be the orthogonal projection of $\Phi_t(v)$ to $\cN_{\phi_t(x)}$. In this way one defines a flow $\psi_t=\psi^X_t$ on the normal bundle $\cN_M$, which is called the {\em linear Poincar\'e flow}.
Note that the normal bundle is not defined at singularities. Hence the base of the normal bundle is non-compact for singular flows. We present below a compactification given by \cite{LGW}.

Define the fiber bundle $G^1=\bigcup\limits_{x\in M}Gr(1,T_xM)$, where $Gr(1,T_xM)$ is the Grassmannian of lines in the tangent space $T_xM$ through the origin. Let $\beta:G^1\to M$ be the corresponding bundle projection that sends a line $l\in Gr(1,T_xM)$ to $x$. The tangent flow $\Phi_t$ induces a flow $\hat{\Phi}_t$ on $G^1$ defined by $\hat{\Phi}_t(\langle v\rangle)=\langle\Phi_t(v)\rangle$, where $v\in TM$ is a nonzero vector and $\langle v\rangle$ is the linear subspace spanned by $v$.

Let $\xi: TM\to M$ be the bundle projection of the tangent bundle. One then defines a vector bundle with base $G^1$ as
\[\beta^*(TM)=\{(l,v)\in G^1\times TM: \beta(l)=\xi(v)\}.\]
The corresponding bundle projection $\iota$ sends every vector $(l,v)\in\beta^*(TM)$ to $l$.
The tangent flow induces also a flow $\tilde{\Phi}_t$ on $\beta^*(TM)$:
\[\tilde{\Phi}_t(l,v)=(\hat{\Phi}_t(l),\Phi_t(v)),\quad \forall (l,v)\in\beta^*(TM).\]
The flow $\tilde{\Phi}_t$ will be called the {\em extended tangent flow}.

For any $l\in G^1$ such that $\beta(l)=x$, there is a natural identification between $T_xM$ and $\{l\}\times T_xM$.
Thus we define the {\em extended normal bundle} as
\[\wt{\cN}=\{(l,v)\in \beta^*(TM): v\perp l\}.\]
We also define the extended normal bundle over any nonempty subset $\Delta\subset G^1$:
\[\wt{\cN}_{\Delta}=\{(l,v)\in\iota^{-1}(\Delta): v\perp l\}.\]
Finally, we define the {\em extended linear Poincar\'e flow} $\tilde{\psi}_t$ on $\wt{\cN}$ as follows:
\[\tilde{\psi}_t(l,v)=\tilde{\Phi}_t(l,v)-\langle \Phi_t(u),\Phi_t(v)\rangle\cdot\tilde{\Phi}_t(l,u),\]
where $u\in l$ is a unit vector, and $\langle \cdot,\cdot\rangle$ stands for the inner product on the tangent bundle.

\begin{Remark}\label{rem.normal-bundle}
   For any $x\notin\Sing(X)$, let $l=l^X_x$, then $\wt{\cN}_l$ can be naturally identified with $\cN_x$ and $\tilde{\psi}_t|_{\wt{\cN}_l}$ can be naturally identified with $\psi_t|_{\cN_x}$.
\end{Remark}

We will also consider invariant measures on $G^1$. Let $\mu$ be any Borel measure on $G^1$, the bundle projection $\beta:G^1\to M$ induces a measure $\beta_*\mu$ on $M$ such that $(\beta_*\mu)(A)=\mu(\beta^{-1}(A))$ for any Borel set $A\subset M$. When $\mu$ is $\hat{\Phi}_t$-invariant, the induced measure $\beta_*\mu$ is $\phi_t$-invariant.

\subsection{Fundamental limit and dominated splitting}

Given a chain recurrent set $\Gamma$ of the vector field $X$. Let $X_n$ be a sequence of $C^1$ vector fields such that $X_n\to X$ in the $C^1$ topology. Suppose there exists a periodic orbit $\gamma_n$ of $X_n$ such that $\gamma_n$ converges to a compact subset of $\Gamma$ in the Hausdorff topology, then the sequence of pairs $(\gamma_n, X_n)$ is called a {\em fundamental sequence} of $\Gamma$, and will be denoted as $(\gamma_n,X_n)\hookrightarrow (\Gamma,X)$. The sequence $(\gamma_n,X_n)$ is called an {\em $i$-fundamental sequence} if $\ind(\gamma_n)=i$ for all $n$ large enough. When there is no possible ambiguity, we will simply say that $\gamma_n$ is a fundamental sequence of $\Gamma$. We denote by
$\cF(\Gamma)$ the limit of directions of all fundamental sequences of $\Gamma$, {\it i.e.}
\[\cF(\Gamma)=\{l\in G^1: \exists (\gamma_n,X_n)\hookrightarrow (\Gamma,X),\ p_n\in \gamma_n, \ \text{such that}\ l^{X_n}_{p_n} \to l\}.\]
As an immediate consequence of the definition, the map $\cF(\cdot)$ is upper semi-continuous in the following sense.
\begin{Lemma}\label{lem.semi-cont}
  Suppose $\Gamma$ is a chain transitive set of $X$. Then for any neighborhood $\cB$ of $\cF(\Gamma)$ in $G^1$, there exist a neighborhood $U$ of $\Gamma$ and $C^1$ neighborhood $\cU$ of $X$ such that for any $Y\in\cU$ and any chain transitive set $\Gamma_Y$ of $Y$ contained in $U$, it holds $\cF(\Gamma_Y)\subset \cB$.
\end{Lemma}
%\begin{proof}
%  Arguing by contradiction, assume there exists a neighborhood $\cB_0$ of $\cF(\Gamma)$ and a sequence of vector fields $X_n$ converging to $X$ in the $C^1$ topology, together with a sequence of chain transitive sets $\Gamma_n=\Gamma_{X_n}$ converging to a compact subset of $\Gamma$ in the Hausdorff topology, such that $\cF(\Gamma_n)\not\subset \cB_0$ for all $n$. By choosing a subsequence, assume that $\cF(\Gamma_n)$ converges in the Hausdorff topology to $\cB_1$. Then $\cB_1\setminus \cB_0\neq\emptyset$.
%  By definition of the map $\cF(\cdot)$, there exists a fundamental sequence $(\gamma_n,X_n)\hookrightarrow (\Gamma,X)$ such that $\cF(\gamma_n)$ accumulates on a point $l\in \cB_1\setminus\cB_0$. This is a contradiction since $l\in \cF(\Gamma)\subset \cB_0$.
%\end{proof}
We also denote by $B(\Gamma)$ the limit set of directions of regular points contained in $\Gamma$, {\it i.e.}
\[B(\Gamma)=\{l\in G^1: \exists x_n\in \Gamma\setminus\Sing(X), \ \text{such that}\ l^X_{x_n} \to l\}.\]
Note that for any $l\in\cF(\Gamma)$, if $x=\beta(l)$ is a regular point, then $l=l^X_x$. This also holds for $B(\Gamma)$. Thus $\cF(\Gamma)\setminus B(\Gamma)$ is contained in the tangent spaces of singularities.

A hyperbolic singularity $\rho$ is called {\em Lorenz-like} if there is a partially hyperbolic splitting of the tangent bundle $T_{\rho}M=E^{ss}_{\rho}\oplus E^{cu}_{\rho}$ such that $E^{cu}_{\rho}$ is sectionally expanded and it decomposes further into a dominated spitting $E^c_{\rho}\oplus E^u_{\rho}$ with $\dim E^c_{\rho}=1$ and $\Phi_t|_{E^c_{\rho}}$ is contracting. For a Lorenz-like singularity $\rho$, denote by $W^{ss}(\rho)$ its strong stable manifold tangent to $E^{ss}_{\rho}$ at the singularity.
\begin{Lemma}\label{lem.loren-like-singularity}
  Let $C(\rho,X)$ be a chain recurrence class of $X$ containing a Lorenz-like singularity $\rho$ such that $C(\rho,X)\cap(W^{ss}(\rho)\setminus\{\rho\})=\emptyset$. %Then $\cF(C(\rho,X))\cap T_{\rho}M\subset E^{cu}_{\rho}$. Moreover,
  Then there exists a $C^1$ neighborhood $\cU$ of $X$ such that for any $Y\in\cU$, it holds $C(\rho_Y,Y)\cap (W^{ss}(\rho_Y,Y)\setminus\{\rho_Y\})=\emptyset$ and $\cF(C(\rho_Y,Y))\cap T_{\rho_Y}M\subset E^{cu}_{\rho_Y}$.\footnote{Here, to simplify notations, we have identified $T_{\rho_Y}M$ with $Gr(1,T_{\rho_Y}M)$, and $E^{cu}_{\rho_Y}$ with $Gr(1,E^{cu}_{\rho_Y})$. We hope that such abuses of notations shall not cause much confusion.}
\end{Lemma}
\begin{proof}
  Note that the properties defining a Lorenz-like singularity is $C^1$ robust. Since the singularity $\rho$ is Lorenz-like, there is a $C^1$ neighborhood $\cU$ of $X$ such that for any $Y\in\cU$, the continuation $\rho_Y$ is also Lorenz-like. Moreover, by the upper semi-continuity of chain recurrence class and the continuity of local strong stable manifolds, one deduces that $C(\rho_Y,Y)\cap (W^{ss}(\rho_Y,Y)\setminus\{\rho_Y\})=\emptyset$ for any $Y\in \cU$ (shrinking $\cU$ if necessary). Then, a similar argument as in the proof of [LGW, Lemma 4.4] gives $\cF(C(\rho_Y,Y))\cap T_{\rho_Y}M\subset E^{cu}_{\rho_Y}$.
\end{proof}

A fundamental sequence $(\gamma_n,X_n)$ of $\Gamma$ is said to {\em admit an index $i$ dominated splitting} if there exist a $\psi^{X_n}_t$-invariant splitting $\cN_{\gamma_n}=\cN_n^{cs}\oplus \cN_n^{cu}$ ($n\in\mathbb{N}$) of the normal bundle and constants $T>0$, $n_0>0$ such that for any $t>T$ and $n>n_0$, it holds $\dim \cN_n^{cs}=i$  and
\[\|\psi^{X_n}_t|_{\cN_n^{cs}(x)}\|\cdot\|\psi^{X_n}_{-t}|_{\cN_n^{cu}(\phi^{X_n}_t(x))}\|<1/2.\]
Identifying the extended normal bundle $\wt{\cN}_{B(\gamma_n)}$ with $\cN_{\gamma_n}$ (see Remark \ref{rem.normal-bundle}), we see that $\wt{\cN}_{B(\gamma_n)}$ also admits a dominated splitting with respect to $\tilde{\psi}_t^{X_n}$ and of the same index. Thus, the following lemma is a consequence of continuity of domination.

\begin{Lemma}\label{lem.dom-equivalence}
$\cF(\Gamma)$ admits a dominated splitting of index $i$ with respect to $\tilde{\psi}^X_t$ if and only if every fundamental sequence of $\Gamma$ admits an index $i$ dominated splitting (with the same domination constant $T$).
\end{Lemma}

\section{The example: construction and robust properties}\label{sect.construction}
In this section we construct the example and claim some key properties of the example. With these properties, a proof of Theorem A is given at the end of this section.

\subsection{The Lorenz attractor}\label{sect.lorenz}
Let us begin with the geometric model of the Lorenz attractor, see e.g. \cite{GH,GW,BDV}. Let $(x_1,x_2,x_3)$ be a Cartesian coordinate system in $\RR^3$. Let $B^3$ be a ball in $\mathbb{R}^3$ centered at the origin $O=(0,0,0)$. We shall consider a $C^1$ vector field $X^0$ on $\RR^3$ such that it is transverse to the boundary of $B^3$ and that the ball $B^3$ is an attracting region, i.e. $\phi_t^{X^0}(x)\in B^3$ for any $x\in B^3$ and $t>0$. Moreover, the following properties are assumed (see \cite[Section 3]{AP}):
\begin{enumerate}[label=(P\arabic*)]
	\item (Lorenz-like singularity) The origin $O$ is a hyperbolic singularity of stable index 2 such that its local stable manifold $W^s_{loc}(O)$ coincides with the $x_1x_3$-plane, and its local unstable manifold $W^u_{loc}(O)$ coincides with the $x_2$-axis. Moreover, the two-dimensional stable subspace $E^s_O$ decomposes into a dominated splitting $E^s_O=E^{ss}_O\oplus E^{c}_O$. Assume that the local strong stable manifold (tangent to $E^{ss}_O$) coincides with the $x_1$-axis. The three Lyapunov exponents at the singularity are \[\lambda^s=\log\|D\Phi^{X^0}_1|_{E^{ss}_O}\|,\quad \lambda^c=\log\|D\Phi^{X^0}_1|_{E^{c}_O}\|,\quad \text{and}\ \lambda^u=\log\|D\Phi^{X^0}_1|_{E^{u}_O}\|.\]
Here, $E^u_O$ is the unstable subspace of $O$. One has $\lambda^s<\lambda^c<0<\lambda^u$. We assume that $\lambda^c+\lambda^u>0$, which implies sectional expanding property of the subspace $E^{cu}_O=E^{c}_O\oplus E^u_O$.

    \item (Cross-section and the first return map) The square $\Sigma_0=\{(x_1,x_2,1): -1\leq x_1,x_2\leq 1\}$ is a cross-section of the flow $\phi^{X^0}_t$, meaning that the vector field at every point of $\Sigma_0$ is transverse to $\Sigma_0$. For simplicity, we assume that the vector field is {\em orthogonal} to $\Sigma_0$ and $\|X^0(x)\|=1$ for any $x\in\Sigma_0$. The cross-section $\Sigma_0$ intersects $W^s_{loc}(O)$ at a line segment $L_0=\{(x_1,0,1): -1\leq x_1\leq 1\}$, which cuts $\Sigma_0$ into a left part $\Sigma_0^-$ and a right part $\Sigma_0^+$, $\Sigma_0\setminus L_0=\Sigma_0^-\cup \Sigma_0^+$. There is a first return map $R_0: \Sigma_0\setminus L_0\to \Sigma_0$ such that the images $R_0(\Sigma_0^-)$ and $R_0(\Sigma_0^+)$ are each a square pinched at one end and contained in the interior of $\Sigma_0$, as shown in Figure \ref{fig.1}. Let us assume that for any $x\in \Sigma_0\setminus L_0$, the time $t_x>0$ satisfying $\phi^X_{t_x}(x)=R_0(x)$ for its first return is at least 2, i.e. $t_x>2$.
        \begin{figure}[htbp]
          \centering
          \includegraphics[width=.5\textwidth]{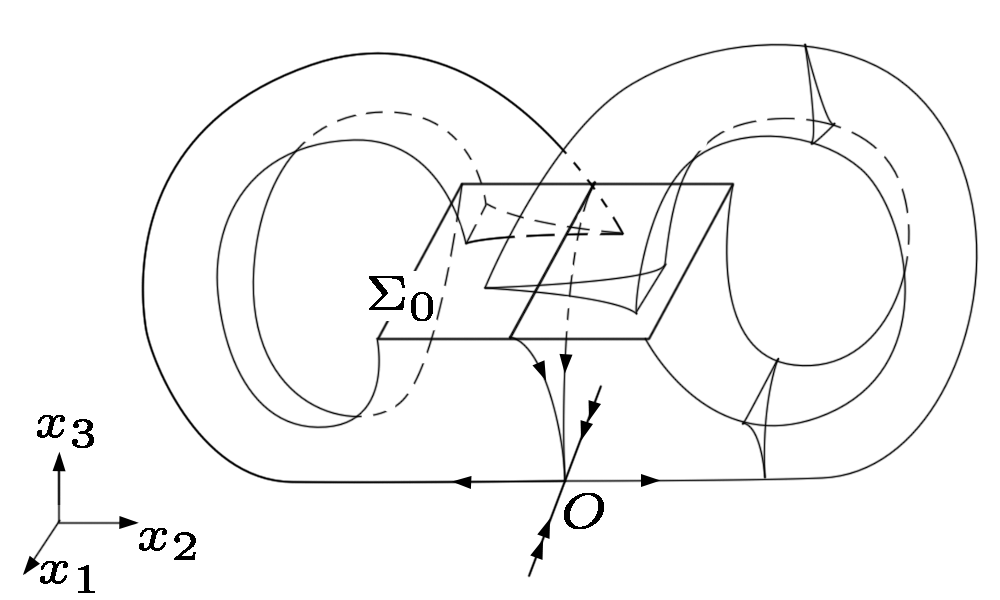}
          \caption{The Lorenz attractor }\label{fig.1}
        \end{figure}
    \item (Cone field on the cross-section) For each $\alpha>0$, there is a cone $C_{\alpha}(x)$ at the point $x=(0,0,1)$:
        \[C_{\alpha}(x)=\{(x_1,x_2)\in T_{x}\Sigma_0 : \|x_1\|\leq \alpha\|x_2\|\}.\]
        Here, we have identified $T_x\Sigma_0$ with $\Sigma_0$ and use the same coordinates $x_1,x_2$. By translating $C_{\alpha}(x)$ to every other point on $\Sigma_0$, one defines a cone field $C_{\alpha}$ on $\Sigma_0$. We assume that $R_0$ preserves the cone $C_{\alpha}$ with $\alpha=1$. More precisely, for any $x\in \Sigma_0\setminus L_0$ and any vector $u\in C_1(x)$, we assume $DR_0(u)\in C_{1/2}(R_0(x))$.
    \item (The attractor) Let $\Lambda_0=\Cl(\phi^{X^0}_{t\in\mathbb{R}}(\Lambda_{\Sigma_0}))$, where $\Lambda_{\Sigma_0}=\cap_{n\geq 0}\overline{R_0}^n(\Sigma_0)$ and $\overline{R_0}^n(\Sigma_0)$ is the closure of $n$-th return of $\Sigma_0$, ignoring $L_0$ for each return.
Then $\Lambda_0$ contains $O$ and is the unique attractor in $B^3$, with an attracting neighborhood $U_{\Lambda_0}\subset B^3$ that contains $\Sigma_0$.
	
    \item (Singular hyperbolicity) The attractor $\Lambda_0$ admits a singular hyperbolic splitting $T_{\Lambda_0}\RR^3=E^{ss}\oplus E^{cu}$, where $E^{ss}$ is uniform contracted and $E^{cu}$ is sectionally expanded. Precisely, we assume that for any two dimensional subspace $S\subset E^{cu}_x$, $x\in\Lambda_0$, it holds
\begin{equation}\label{eq.sect-exp}
 \left|\det(\Phi^{X^0}_t|_S)\right|>e^{\gamma t},\quad \forall t\geq 1,
\end{equation}
where $\gamma>0$ is a constant. The constant $\gamma$ will be assumed to be large so that the inequality \eqref{eq.sect-exp} implies expanding property for the first return map $R_0$: there exists $\rho>1$ such that for any $x\in \Lambda_0\cap(\Sigma_0\setminus L_0)$, $v\in E^{cu}_x\cap (T_x\Sigma_0\setminus\{0\})$, it holds
\[\|DR_0(v)\|>\rho\|v\|.\]
As the flow direction $X^0$ is invariant and can not be uniformly contracted, it is contained in the subbundle $E^{cu}$, i.e. ${X^0}(x)\in E^{cu}_x$ for every $x\in\Lambda_0\setminus\Sing(X)$ (see \cite[Lemma 3.4]{BGY}). %Since on the cross-section $\Sigma\setminus L_0$ we have $DR_0(C_1)\subset C_{1/2}$, it holds that $E^{cu}_x\in C_{1/2}(x)\cap T_x\Sigma$ for any $x\in \Sigma\setminus L_0$.
At the singularity $O$, we have $E^{cu}_O=E^{c}_O\oplus E^u_O$. Define
\[\lambda^s_0=\log\|D\Phi^{X^0}_1|_{E^{ss}_{\Lambda_0}}\|.\]
Note that $\lambda^s\leq \lambda^s_0$. We assume $\lambda^s_0<\lambda^c$.

	\item (Stable foliation) The stable foliation $\cW^s(\Lambda_0)=\{W^s(\orb(x)): x\in \Lambda_0\}$ is $C^1$ (see \cite{AM}), which induces a $C^1$ foliation $\cF^s$ on $\Sigma_0$ such that for any $x\in \Sigma_0$, $\cF^s(x)$ is transverse to $C_1(x)$. Since $L_0$ is the intersection of $W^s_{loc}(O)$ with $\Sigma_0$, it is a leaf of $\cF^s$.
    \item (Transitivity) The attractor $\Lambda_0$ is a homoclinic class (see \cite{Ba}), so that $\Lambda_0$ is transitive and there is a periodic orbit $Q_0\subset \Lambda_0$ whose stable manifold is dense in $U_{\Lambda_0}$.
\end{enumerate}

\subsection{The skew product construction}\label{sect.skew-product}

Let $\Omega=B^3\times I$, where $I=[-1,1]$. Let $x$, $s$ be the coordinates on $B^3$ and $I$, respectively. Define on $\Omega$ the vector field $\hat{X}(x,s)=(X^0(x),-\theta s)$, where $\theta>0$ is a constant. As the fiber direction (along $I$) is uniformly contracting, there exists a unique attractor $\Lambda=\Lambda_0\times\{0\}$, which is a trivial construction of the Lorenz attractor in dimension 4. All properties of the 3-dimensional attractor $\Lambda_0$ can be easily generalized to $\Lambda$.
We now modify the vector field $\hat{X}$ along fibers to obtain a vector field $X$ so that $\Lambda$ remains an attractor but not singular hyperbolic. This is why we say the example is derived from the Lorenz attractor.
The modification will be done in a small neighborhood of a periodic orbit in $\Lambda$.

Let $P_0\subset \Lambda_0$ be a periodic orbit of $X^0$ other than $Q_0$ such that it is homoclinically related to $Q_0$. Let $P=P_0\times\{0\}$ and $Q=Q_0\times \{0\}$. Consider a small neighborhood $V_P$ of $P$ contained in $U_{\Lambda_0}\times(-1,1)$ such that $\overline{V_P}\cap (Q\cup\{\sigma\})=\emptyset$, where $\sigma=(O,0)$.
Let $\eta$ be a $C^{\infty}$ function on $M_0$ that satisfies the following conditions:
\begin{itemize}
    \item $0\leq \eta(x,s)\leq 1$, for all $(x,s)\in \Omega$;
    \item $\eta$ is supported on $V_P$, i.e. $\eta(x,s)=0$ for any $(x,s)\not\in V_P$;
	\item $\eta(x,s)=1$ if and only if $(x,s)\in P$.
\end{itemize}
The vector field $X$ on $\Omega$ is defined as the following:
\begin{equation}\label{eq.vector-field}
X(x,s)=(X^0(x), -\theta s(1-\eta(x,s))).
\end{equation}
Observe that the dynamics of the vector field is contracting along the fibers, except only for $P$ where it is neutral. Also,  $U_{\Lambda}=U_{\Lambda_0}\times (-1,1)$ is an attracting region of $X$ and the maximal invariant set $\Lambda=\Lambda_0\times\{0\}$ in $U_{\Lambda}$ is an attractor. The attractor contains a unique singularity $\sigma=(O,0)$. The tangent space at $\sigma$ admits a $\Phi^X_t$-invariant splitting $T_{\sigma}\Omega=T_{\sigma}B^3\oplus \RR^I$, in which we identify $B^3$ with $B^3\times\{0\}$ and $\RR^I$ is the subspace corresponding to fiber. Thus, the Lyapunov exponents of $\sigma$ are $\lambda^s, \lambda^c, \lambda^u$ and $-\theta$. Note that $-\theta$ is the Lyapunov exponent of the singularity along the fiber direction.
We assume further that
\begin{equation}\label{eq.eigenvalues}
	\lambda_0^s<-\theta<\lambda^c.
\end{equation}
This implies in particular that there is a dominated splitting of the stable subspace $E^s_{\sigma}=E^{ss}\oplus\RR^I\oplus E^{cs}$.
The following properties can be easily verified:
\begin{enumerate}[label=(C\arabic*)]
  \item The cube $\Sigma=\Sigma_0\times I$ is a cross section of the flow $\phi^X_t$. The local stable manifold of $\sigma$ cuts $\Sigma$ along a two-dimensional disk $L=L_0\times I$. So $\Sigma\setminus L$ consists of a left part $\Sigma^-=\Sigma_0^-\times I$ and a right part $\Sigma^+=\Sigma_0^+\times I$. Let $R:\Sigma\setminus L\to \Sigma$ be the first return map. Then $R(\Sigma^-)$ is a cube pinched at one end and contained in the interior of $\Sigma$, and similarly for $R(\Sigma^+)$. For any $p\in \Sigma\setminus L$, let $t_p>0$ be the smallest time that $\phi^X_{t_p}(p)=R(p)$. By assumption on $R_0$, we have $t_p>2$ for any $p\in \Sigma\setminus L$.
  \item $Q$ is a hyperbolic periodic orbit contained in $\Lambda$ with stable index 2. The stable manifold $W^s(Q,X)$ is dense in $U_{\Lambda}$. This follows from the fact that the stable manifold of the periodic orbit $Q_0$ is dense in $U_{\Lambda_0}$ and that the tangent flow along the fibers is topologically contracting. Moreover, $W^s(Q,X)$ intersects $\Sigma$ along a dense family of 2-dimensional $C^1$ disks, which are leaves of the foliation $\cF^s\times I$. Here, $\cF^s$ is the foliation on $\Sigma_0$ as in (P6). %The foliation $\cF^s\times[-1,1]$ is transverse to $E^{cu}$.
  \item $P\subset \Lambda$ is a non-hyperbolic periodic orbit of $X$: it has a zero exponent along the fiber direction. Nonetheless, $P$ has a 3-dimensional topologically stable manifold and its strong stable manifold $W^{ss}(P,X)$ (2-dimensional) has nonempty intersection with the unstable manifold of $Q$. Also, its unstable manifold has a transverse intersection with the stable manifold of $Q$.
\end{enumerate}

\subsection{Robust properties of the example}

As a first result, we show that $\Lambda$ admits a partially hyperbolic splitting.
\begin{Lemma}\label{lem.ph-splitting}
 There exists a partially hyperbolic splitting $T_{\Lambda}\Omega=E^{ss}\oplus F$ for the tangent flow with $\dim E^{ss}=1$.
\end{Lemma}
\begin{proof}
Since $T_{\Lambda}\Omega=T_{\Lambda_0}B^3\oplus \RR^I$ and $T_{\Lambda_0}B^3$ admits a partially hyperbolic splitting $E^{ss}\oplus E^{cu}$, one obtains an invariant splitting $T_{\Lambda}\Omega=E^{ss}\oplus F$, where $F=E^{cu}\oplus \RR^I$. Since $E^{ss}$ is already dominated by $E^{cu}$, equation \eqref{eq.eigenvalues} implies that $E^{ss}$ is dominated by $F$. Moreover, $\Phi^X_t|_{E^{ss}}$ is contracting as it can be identified with $\Phi^{X_0}_t|_{E^{ss}}$. Hence $T_{\Lambda}\Omega=E^{ss}\oplus F$ is a partially hyperbolic splitting.
\end{proof}

Note that the splitting $F=E^{cu}\oplus \RR^I$ is invariant but not dominated. This is because the bundle $\RR^I$ is neutral on the periodic orbit $P=P_0\times\{0\}$ and cannot be dominated by the bundle $E^{cu}$ which contains the flow direction, and vice versa. Also, $F$ is not sectionally expanded. Hence $\Lambda$ is not singular hyperbolic. Nonetheless, the invariance of the splitting $T_{\Lambda}\Omega=E^{ss}\oplus E^{cu}\oplus \RR^I$ suggests also an invariant splitting of the extended normal bundle with one-dimensional subbundles.
\begin{Proposition}\label{prop.away-tangencies}
    The extended normal bundle $\wt{\cN}_{\cF(\Lambda)}$ admits a dominated splitting $\cN^{ss}\oplus\cN^I\oplus\cN_2$ with one-dimensional subbundles. Consequently, $X|_{\Lambda}$ is $C^1$ locally away from homoclinic tangencies.
\end{Proposition}
Here, a chain recurrence class is called {\em locally away from homoclinic tangencies} if there exists a neighborhood $U$ of the class and a $C^1$ neighborhood $\cU$ of $X$ such that any $Y\in\cU$ admits no homoclinic tangency in $U$.

Also, the sectionally expanding subbundle $E^{cu}$ in the splitting $T_{\Lambda}\Omega=E^{ss}\oplus E^{cu}\oplus \RR^I$ will be shown to exist robustly in the following sense.
\begin{Proposition}\label{prop.sect-exp}
  There exist a neighborhood $U\subset U_{\Lambda}$ of $\Lambda$ and a neighborhood $\cU$ of $X$ such that for any $Y\in \cU$ and any chain recurrence class $C_Y\subset U$, there is a continuous 2-dimensional subbundle $E^Y$ of $T_{C_Y}\Omega$ containing the flow direction such that it is invariant for $\Phi^Y_t$ and sectionally expanded. Moreover, if $C_Y=C(\si_Y,Y)$ is the chain recurrence class containing $\si_Y$, then $E^Y$ varies upper semi-continuously with respect to $Y\in\cU$.
\end{Proposition}

The proof of Proposition \ref{prop.away-tangencies} and Proposition \ref{prop.sect-exp} will be given in Section \ref{sect.away-from-ht} and Section \ref{sect.sectional-expanding}, respectively.

By construction, the chain recurrence class $\Lambda$ is an attractor containing the unique singularity $\sigma=(O,0)$. We will show that for any $C^1$ vector field $Y$ close enough to $X$, the chain recurrence class $C(\sigma_Y,Y)$ is nontrivial and Lyapunov stable. Moreover, it is the only Lyapunov stable chain recurrence class in $U_{\Lambda}$, where $\sigma_Y$ is the continuation of the singularity. Precisely, we have

\begin{Proposition}\label{prop.ls-class}
	There exists a $C^1$ neighborhood $\cU$ of $X$ such that for any $Y\in \cU$, the chain recurrence class $C(\si_{Y},Y)$ is the unique Lyapunov stable chain recurrence contained in $U_{\Lambda}$. Moreover, $C(\si_{Y},Y)$ contains the periodic orbit $Q_Y$, which is the continuation of $Q$.
\end{Proposition}
The proof of Proposition \ref{prop.ls-class} will be given in Section \ref{sect.ls}.

\subsection{Robust heterodimensional cycles}\label{sect:robust-cycles}

Let us continue the construction of the example. We will prove the following result.

\begin{Lemma}\label{lem.hc}
  In every $C^1$ neighborhood of $X$ there is an open subset $\cV$ of vector fields such that for each $Y\in\cV$ the chain recurrence class $C(\sigma_{Y},Y)$ contains a robust heterodimensional cycle.
\end{Lemma}
\begin{Remark}
  By the lemma, the vector field $X$ can be $C^1$ approximated by vector fields with robust heterodimensional cycles.
  Also, it follows from the construction that $X$ can be $C^1$ approximated by vector fields with singular hyperbolic attractors. Therefore, $X$ is on the boundary between singular hyperbolicity and robust heterodimensional cycles.
\end{Remark}

The proof of Lemma \ref{lem.hc} is essentially contained in \cite[Section 4]{BD}. We say that a periodic orbit is a {\em saddle-node} if its normal bundle has an (orientable) one-dimensional center along which the Lyapunov exponent is zero, and all other Lyapunov exponents are nonzero.
For the vector field $X$, the periodic orbit $P$ is a saddle-node with a one-dimensional center in the fiber direction, along which the Lyapunov exponent is zero. The saddle-node $P$ has a strong stable manifold $W^{ss}(P,X)$ contained in $B^3\times\{0\}$ (as in the 3-dimensional Lorenz attractor). Also its unstable manifold $W^u(P,X)$ is contained in $B^3\times\{0\}$. The saddle-node $P$ has a {\em strong homoclinic intersection}: there exists $r\notin P$ which belongs to the intersection $W^{ss}(P,X)\cap W^u(P,X)$. Moreover, the intersection is {\em quasi-transverse}: $\dim (T_rW^{ss}(P,X)\oplus T_rW^u(P,X))=3<\dim \Omega=4$.

\begin{Lemma}[{\cite[Theorem 4.1]{BD}}]\label{lem.robust-cycles}
  Let $X$ be a $C^1$ vector field with a quasi-transverse strong homoclinic intersection associated to a saddle-node. Then in every $C^1$ neighborhood of $X$ there is an open set of vector fields exhibiting robust heterodimensional cycles.
\end{Lemma}

It follows from the discussions above and Lemma \ref{lem.robust-cycles} that there exists vector field $Y$ arbitrarily close to $X$ which exhibits a robust heterodimensional cycle. Thus, to prove Lemma \ref{lem.hc}, it remains to show that the robust heterodimensional cycle is contained in $C(\si_{Y},Y)$. For this, we need to slightly change the proof of Lemma \ref{lem.robust-cycles} ({\cite[Section 4]{BD}}), in which a key ingredient is the creation of robust heterodimensional cycles through {\em blenders}. Roughly speaking, a blender is a hyperbolic set whose stable (or unstable) set behaves as if it is one dimensional higher than it actually is. One can refer to \cite[Chapter 6]{BDV} for more discussions on blenders.

\begin{proof}[Proof of Lemma \ref{lem.hc}]
  We have seen that the periodic orbit $P$ is a saddle-node associated to which there is a quasi-transverse strong homoclinic intersection. In fact, we can take two quasi-transverse strong homoclinic intersections $z$ and $w$ associated to $P$ such that their orbits are disjoint. Following the proof of Lemma \ref{lem.robust-cycles} (\cite[Section 4]{BD}), one can make arbitrarily small $C^1$ perturbations in a neighborhood of $\Cl(\orb(z,X)\cup\orb(w,X))$ (containing $P$) to obtain a vector field $Y$ having a blender $\Lambda'$, together with a hyperbolic periodic orbit $P_{\delta}$ close to $P_Y=P$ such that:
  \begin{itemize}
	\item the blender $\Lambda'$ is a hyperbolic set of stable index $1$ containing the periodic orbit $P_Y$;
	\item the periodic orbit $P_{\delta}$ has stable index $2$, and its stable manifold intersect transversely the unstable manifold of $P_Y$;
	\item the unstable manifold of $P_{\delta}$ meets robustly the stable set of the blender $\Lambda'$ (by the property of the blender);
	\item thus, the vector field $Y$ has a robust heterodimensional cycle associated with the hyperbolic set $\Lambda'\supset P_Y$ and the hyperbolic saddle $P_{\delta}$.
  \end{itemize}
  We can assume that $Y$ is close enough to $X$ so that by Proposition \ref{prop.ls-class}, $C(\si_Y,Y)$ robustly contains the periodic orbit $Q_Y$. Thus, to guarantee that $C(\si_Y,Y)$ contains the robust heterodimensional cycle, we make sure that $Q_Y$ and $P_Y$ are robustly contained in the same chain recurrence class. A slightly different procedure should be followed when making the perturbations to obtain the robust heterodimensional cycle.

  Note that $W^s(Q,X)$ intersects $W^u(P,X)$ transversely. As $Y$ is $C^1$ close to $X$, we can assume that $W^s(Q_Y,Y)$ and $W^u(P_Y,Y)$ intersects robustly. Now if $W^u(Q_Y,Y)\pitchfork W^s(P_{\delta},Y)\neq\emptyset$, the Inclination Lemma would imply that $W^u(Q_Y,Y)$ accumulates on $W^u(P_{\delta},Y)$ and hence meets transversely the characteristic region of the blender $\Lambda'$. By the property of blender and since $W^s(Q_Y,Y)$ intersects $W^u(P_Y,Y)$ robustly, we would obtain that $Q_Y$ and $P_Y$ are robustly contained in the same chain recurrence class.

  \begin{figure}[htbp]
  \centering
  \subcaptionbox{Blow up to twins}{
    \includegraphics[width=0.21\linewidth]{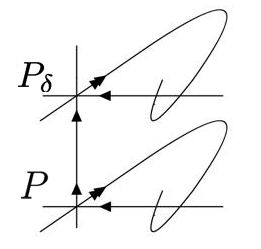}
    }
  \qquad
  \subcaptionbox{Blow up to triplets}{
    \includegraphics[width=0.2\linewidth]{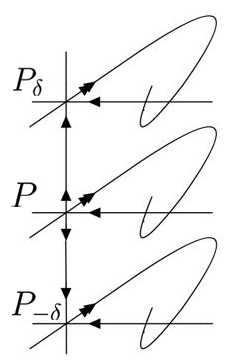}
    }
  \caption{DA-type surgery at the saddle-node $P$}\label{fig.2}
\end{figure}

  Therefore, when perturbing to obtained the periodic orbit $P_{\delta}$ we need to make sure that $W^u(Q_Y,Y)\pitchfork W^s(P_{\delta},Y)\neq\emptyset$. Recall that in \cite[Section 4]{BD}, $P_{\delta}$ is obtained as a twin of $P$ by blowing up the saddle-node $P$ along the neutral center, as shown in Figure \ref{fig.2}(a). In stead of obtaining only a twin $P_{\delta}$, we blow up the saddle-node to obtain a set of triplet saddles, $P$, $P_{\delta}$ and $P_{-\delta}$, as shown in Figure \ref{fig.2}(b). See also \cite{Ma} for a similar construction. The saddles $P_{\delta}$ and $P_{-\delta}$ are both of index 2, while $P$ has index 1. Note that for the vector $X$, there is a quasi-transverse intersection between $W^u(Q,X)$ and $W^{ss}(P,X)$. After the blow-up, we can assume without loss of generality that $W^u(Q_Y,Y)\pitchfork W^s(P_{\delta},Y)\neq\emptyset$. Then the construction continues as in \cite[Section 4]{BD}. This ends the proof of Lemma \ref{lem.hc}.
\end{proof}

\subsection{Proof of Theorem A}

Let $M$ be any 4-dimensional closed Riemannian manifold. One considers a gradient-like vector field $Z_0$ on $M$, which has at least a sink. Consider a local chart around a sink and a ball $B^4$ in the chart, centered at the sink. By shrinking the local chart, we assume that the ball $B^4$ is an attracting region of the flow. In particular, the vector field $Z_0$ is transverse to $\partial B^4$ and inwardly pointing. By changing the metric, we assume that $\frac{1}{2}B^4$ contains $\Omega=B^3\times [-1,1]$. Let $X$ be a vector field on $\Omega$ as constructed in Section \ref{sect.lorenz} and Section \ref{sect.skew-product}. One then modifies the vector field $Z_0$ in a neighborhood of $\frac{1}{2}B^4$ to obtain a $C^1$ vector field $Z$ such that it coincides with $X$ on $\Omega\subset\frac{1}{2}B^4$. The modification shall be taken outside a neighborhood of $M\setminus B^4$, thus $Z$ coincides with $Z_0$ in a neighborhood of $M\setminus B^4$.
Moreover, since the vector field on $\Omega$ is transverse to each smooth component of $\partial \Omega$ and inwardly pointing, one can require that the obtained vector field $Z$ has no recurrence in $B^4\setminus \Omega$. Then the chain recurrence set of $Z$ is composed of finitely many hyperbolic critical elements and the maximal invariant set in $\Omega$.

The following properties can be verified.
\begin{enumerate}[label=(\arabic*)]
  \item There is nontrivial chain recurrence class $C(\si,Z)$ associated to a hyperbolic singularity $\si\in\Omega$.
  \item There exist a neighborhood $U\subset \Omega$ of $C(\si,Z)$ and a neighborhood $\cU$ of $Z$ such that for any $Y\in \cU$, the continuation $C(\si_Y,Y)$ is well-defined and contained in $U$. In particular, $\si_Y$ is the only singularity in $U$. Moreover, since $C(\si,Z)$ is an attractor, $U$ can be chosen to be an attracting neighborhood of $C(\si,Z)$. By reducing $\cU$, we assume that $U$ remains an attracting region for any $Y\in\cU$ and the chain recurrent set of $Y$ is composed of finitely many hyperbolic critical elements plus the maximal invariant set in $U$.
  \item For any $Y\in\cU$, there exists no homoclinic tangency in $U$ (Proposition \ref{prop.away-tangencies}). As the chain recurrent set of $Y$ outside $U$ is composed of finitely many hyperbolic critical elements, one concludes that $Y$ is away from homoclinic tangencies.
  \item There is a partially hyperbolic splitting $T_{C(\si_Y,Y)}M=E^{ss}\oplus F$ with $\dim E^{ss}=1$ (Lemma \ref{lem.ph-splitting}). The subbundle $F$ contains a continuous 2-dimensional subbundle $E^Y\subset T_{C(\si_Y,Y)}M$ containing the flow direction such that it is invariant for $\Phi^Y_t$ and sectionally expanded. Moreover, $E^Y$ varies upper semi-continuously with respect to $Y\in\cU$.  (Proposition \ref{prop.sect-exp})
  \item $C(\si_Y,Y)$ is the unique Lyapunov stable chain recurrence class contained in $U$. (Proposition \ref{prop.ls-class})
\end{enumerate}
By Lemma \ref{lem.hc}, there exists an open set $\cV\subset \cU$ such that for any $Y\in\cV$, the chain recurrence class $C(\si_Y,Y)$ exhibits a robust heterodimensional cycle. To finish the proof of Theorem A, it remains to show that for any $Y\in\cV$, the subbundle $E^Y$ is exceptional. Arguing by contradiction, if $E^Y$ is not exceptional, then there exists a dominated splitting $F=E^Y\oplus E'$ or $F=E'\oplus E^Y$. In the first case, since $E^Y$ is sectionally expanded and is dominated by $E'$, one deduces that $E'$ is expanded. It follows that $C(\si_Y,Y)$ is singular hyperbolic, contradicting to the fact that $C(\si_Y,Y)$ contains a robust heterodimensional cycle. In the second case, since $E'$ is dominated by $E^Y$ which contains the flow direction, one deduces that $E'$ should be contracted. Then $C(\si_Y,Y)$ would also be singular hyperbolic, a contradiction again.

This completes the proof of Theorem A.

\section{Away from tangencies: proof of Proposition \ref{prop.away-tangencies}}\label{sect.away-from-ht}

This section is devoted to the proof of Proposition \ref{prop.away-tangencies}. Recall that the statement of Proposition \ref{prop.away-tangencies} contains two parts:
\begin{enumerate}[label=(\alph*)]
  \item $\wt{\cN}_{\cF(\Lambda)}$ admits a dominated splitting $\cN^{ss}\oplus\cN^I\oplus\cN_2$ with one-dimensional subbundles;
  \item the dynamics $X|_{\Lambda}$ is locally away from homoclinic tangencies.
\end{enumerate}
Note first that part (b) follows from part (a) and the next lemma.
%We will use the following result.
\begin{Lemma}[\cite{W02,Go}]\label{lem.away-tang}
	The dynamics $X|_{\Lambda}$ is locally away from homoclinic tangencies if and only if every $i$-fundamental sequence of $\Lambda$ admits an index $i$ dominated splitting.
\end{Lemma}
Part (a) holds if we show that $\cF(\Lambda)=B(\Lambda)$ and $\wt{\cN}_{B(\Lambda)}$ admits a dominated splitting with one dimensional subbundles. These two steps are given below in Lemma \ref{lem.f-seq} and Lemma \ref{lem.dominated-splitting}, respectively.
Note that these results are obtained for the unperturbed vector field $X$ whose dynamics in $\Omega=B^3\times I$ is a skew-product.

\begin{Lemma}\label{lem.f-seq}
$\cF(\Lambda)=B(\Lambda)$. 	
\end{Lemma}

\begin{proof}
	Since $\Lambda$ is a homoclinic class, periodic orbits are dense in $\Lambda$. This implies that $B(\Lambda)\subset \cF(\Lambda)$.
	We need to show that the converse also holds. By definition of the function $\cF(\cdot)$, for any $l\in \cF(\Lambda)$, we have $\beta(l)\in \Lambda$,  and if $\beta(l)=x$ is a regular point, then $l=l^X_x$.
	 Therefore, as $\sigma$ is the only singularity in $\Lambda$, the set $\cF(\Lambda)\setminus B(\Lambda)$ is contained in $T_{\sigma}\Omega$. At the singularity, there is a dominated splitting $T_{\sigma}\Omega= (E^{ss}_{\si}\oplus \RR^I)\oplus E^c_{\si}\oplus E^u_{\si}$, see \eqref{eq.eigenvalues}. Let $W^{ss}(\si)$ be the strong stable manifold of $\si$ that is tangent to $E^{ss}_{\si}\oplus\RR^I$ at $\si$.
By construction, $\Lambda\cap W^{ss}(\si)\setminus\{\si\}=\emptyset$. Thus it follows from Lemma \ref{lem.loren-like-singularity} that $\cF(\Lambda)\cap T_{\si}\Omega \subset E^c_{\si}\oplus E^u_{\si}$.
Note that for the 3-dimensional Lorenz attractor $\Lambda_0$, the intersection of $B(\Lambda_0)$ with $T_OB^3$ coincides with $E^{cu}_O=E^c_O\oplus E^u_O$. Identifying $O$ with $\sigma=(O,0)$ and $\Lambda_0$ with $\Lambda=\Lambda\times\{0\}$, we obtain that  $B(\Lambda)=B(\Lambda_0)=E^c_{\sigma}\oplus E^u_{\sigma}$.
Thus, $\cF(\Lambda)\cap T_{\si}\Omega \subset B(\Lambda)\cap T_{\si}\Omega$. Consequently, $\cF(\Lambda)\subset B(\Lambda)$. Therefore, $\cF(\Lambda)=B(\Lambda)$.
\end{proof}

\begin{Lemma}\label{lem.dominated-splitting}
	$\wt{\cN}_{B(\Lambda)}$ admits a dominated splitting $\cN^{ss}\oplus\cN^I\oplus\cN_2$ with one-dimensional subbundles.
\end{Lemma}

\begin{proof}
	For the Lorenz attractor $\Lambda_0$ in $B^3$, there exists a singular hyperbolic splitting $T_{\Lambda_0}B^3=E^{ss}\oplus E^{cu}$. For any point $x\in\Lambda_0\setminus\{O\}$, one has $l^{X^0}_x\subset E^{cu}_x$, see (P5). In particular, the domination gives a lower bound of the angle between $E^{ss}_x$ and $l^{X^0}_x$, which is uniform for $x\in\Lambda_0\setminus\{O\}$. Let $\cN^{ss}_x$ be the orthogonal projection of $E^{ss}_x$ to $\cN_x$, and let $\cN_2(x)=E^{cu}_x\cap\cN_x$. The domination property of $T_{\Lambda_0}B^3=E^{ss}\oplus E^{cu}$ and the uniform lower bound of the angle between $E^{ss}_x$ and $l^{X^0}_x$ imply that the bundle $\cN^{ss}$ is also dominated by $\cN_2$ for the linear Poincar\'e flow (\cite[Lemma 2.3]{BGY}). Moreover, with the contracting property of $\Phi^{X_0}_t|_{E^{ss}}$ one can show that the bundle $\cN^{ss}$ is also contracted by the linear Poncar\'e flow. Taking limits of the splitting $\cN_{\Lambda_0}=\cN^{ss}\oplus \cN_2$ in the Grassmannian $G^1$, one obtains a dominated splitting $\wt{\cN}_{B(\Lambda_0)}=\cN^{ss}\oplus \cN_2$. Moreover, the subbundle $\cN^{ss}$ is contracting for the extended linear Poincar\'e flow.
	
	Note that the dynamics in $\Omega=B^3\times I$ is a skew-product and $\Lambda=\Lambda_0\times\{0\}$ is an invariant set for $X$. We may assume that the one-dimensional subbundle $\RR^I$ tangent to the fiber $I$ is orthogonal to $T(B^3\times\{0\})$. Identifying $B^3$ with $B^3\times\{0\}$, one can see that $\wt{\cN}_{B(\Lambda)}$ admits an invariant splitting $\cN^{ss}\oplus \cN^I\oplus \cN_2$, where $\cN^{ss}$ and $\cN_2$ are given by the splitting $\wt{\cN}(B(\Lambda_0))=\cN^{ss}\oplus \cN_2$ and $\cN^I=\RR^I$. We need to prove that this splitting is dominated. Following \cite{Ma2}, it suffices to show that for any ergodic invariant measure supported on $B(\Lambda)$, the Lyapunov exponents corresponding to the three subbundles satisfy
	\begin{equation}\label{eq.exponents}
		\eta^{ss}<\eta^{I}<\eta_2.
	\end{equation}
	
	Recall from the construction that there is a dominated splitting $T_{\si}\Omega=E^{ss}\oplus \RR^I\oplus E^c\oplus E^u$, such that the corresponding Lyapunov exponents of the subbundles are
\[\la^s<-\theta<\la^c<\la^u.\]
Moreover, we have defined $\la^s_0=\log\|D\Phi_1^{X^0}|_{E^{ss}_{\Lambda_0}}\|$ which satisfies $\la^s\leq\la^s_0<-\theta$. See Section \ref{sect.skew-product}.

Let $\mu$ be any ergodic invariant measure supported on $B(\Lambda)$. Suppose $\beta_*\mu$ is the Dirac measure at $\sigma$, then $\mu$ is the Dirac measure supported either on $\langle E^c_{\sigma}\rangle$ or on $\langle E^u_{\sigma}\rangle$. In both cases, we have
	\[\eta^{ss}=\lambda^s < -\theta=\eta^{I}.\]
	Then we show that $\eta^{I}<\eta_2$ as follows: if $\mu$ is the Dirac measure on $\langle E^c_{\sigma}\rangle$, then $\eta_2=\lambda^u$; otherwise, $\mu$ is the Dirac measure on $\langle E^u_{\sigma}\rangle$, then $\eta_2=\lambda^c$. In both cases, it follows from equation \eqref{eq.eigenvalues} that
	\[\eta^{I}=-\theta<\lambda^c\leq \eta_2.\]
	Hence \eqref{eq.exponents} holds when $\beta_*\mu$ is the Dirac measure at $\sigma$.

	Now, suppose $\beta_*\mu$ is not the Dirac measure at $\sigma$, then it is a nonsingular ergodic measure for the flow $\phi_t^X$. Since the bundle $\cN^{ss}$ is obtained from $E^{ss}$ by orthogonal projection, the Lyapunov exponent along $\cN^{ss}$ is no larger than $\log\|D\Phi_1^{X^0}|_{E^{ss}_{\Lambda_0}}\|=\lambda^s_0$, i.e.
	\[\eta^{ss}\leq \lambda^s_0.\]
By construction of the example, on every point $(x,0)\in \Lambda$, it holds
\[-\theta\leq \frac{\partial \zeta(x,s)}{\partial s}\bigg|_{s=0}\leq 0,\]
where $\zeta(x,s)=-\theta s(1-\eta(x,s))$ is the last component of $X(x,s)$.
This implies that $-\theta\leq \eta^{I}\leq 0$. Then by the inequality \eqref{eq.eigenvalues}, one has
\[\eta^{ss}\leq \lambda^s_0<-\theta\leq \eta^I\leq 0.\]
Since $E^{cu}$ is sectionally expanding and $\cN_2$ is obtained from $E^{cu}$ by intersecting with the normal bundle, one deduces that the Lyapunov exponent along $\cN_2$ is larger than zero,  i.e. $\eta_2>0$. Therefore, \eqref{eq.exponents} also holds in this case.
	This finishes the proof of the lemma.
\end{proof}

\begin{Remark}\label{rmk.exceptional-bundle}
  From the proof, we can see that $E^{cu}_{\Lambda}=\beta_*(\cF(\Lambda)\oplus \cN_2)$, where the map $\beta_*$ takes each vector $(l,v)\in\beta^*(TM)$ to $v\in TM$.
\end{Remark}

\begin{proof}[Proof of Proposition \ref{prop.away-tangencies}]
  By Lemma \ref{lem.f-seq} and Lemma \ref{lem.dominated-splitting}, the extended normal bundle $\wt{\cN}_{\cF(\Lambda)}$ admits a dominated splitting with 1-dimensional subbundles. Then it follows from Lemma \ref{lem.dom-equivalence} that any fundamental sequence of $\Lambda$ admits an index $i$ dominated splitting, for $i=1$ and $i=2$. Thus Lemma \ref{lem.away-tang} shows that $X|_{\Lambda}$ is locally away from homoclinic tangencies.
\end{proof}

\section{Existence of sectionally expanding subbundle: proof of
Proposition \ref{prop.sect-exp}} \label{sect.sectional-expanding}
By construction, the chain recurrence class $\Lambda$ admits a 2-dimensional invariant bundle $E^{cu}\subset T_{\Lambda}\Omega$ which is sectionally expanding for $\Phi^X_t$. Proposition \ref{prop.sect-exp} asserts that this property also holds for nearby chain recurrence classes for $C^1$ perturbations.

This section gives a proof of Proposition \ref{prop.sect-exp}.
For this purpose, we turn to the dominated splitting $\wt{\cN}_{\cF(\Lambda)}=\cN_1\oplus \cN_2$ given by Proposition \ref{prop.away-tangencies}, where $\cN_1=\cN^{ss}\oplus\cN^I$.
As in Remark \ref{rmk.exceptional-bundle}, we have $E^{cu}_{\Lambda}=\beta_*(\cF(\Lambda)\oplus \cN_2)$, which is a two-dimensional sectionally expanding subbundle for the tangent flow.

\begin{Definition}
	Let $\cE\subset \beta^*(TM)$ be any invariant bundle for the extended tangent flow $\tilde{\Phi}_t$. We say that the bundle $\cE$ {\em projects to a continuous bundle $E\subset TM$} if $E$ is continuous and for any $l\in \iota(\cE)$, it holds $\beta_*(\cE_l)=E_{\beta(l)}$.
\end{Definition}

\begin{Remark}
    If $\cE$ is continuous and it satisfies $\beta_*(\cE_{l_1})=\beta_*(\cE_{l_2})$ for any $l_1,l_2\in \iota(\cE)$ with $\beta(l_1)=\beta(l_2)$, then $E=\beta_*(\cE)$ is a well-defined continuous subbundle of $TM$. Moreover, it is easy to see that $E$ is invariant for the tangent flow $\Phi_t$ if $\cE$ is invariant for the extended tangent flow $\tilde{\Phi}_t$.
\end{Remark}

Proposition \ref{prop.sect-exp} will be proved as a corollary of the following result, which is stated in a slightly more general setting.

\begin{Proposition}\label{prop.sectional-expanding}
Let $C_X$ be a chain recurrence class of a vector field $X\in\xX^1(M)$ such that every singularity in the class, say $\rho$, is Lorenz-like with stable index $i_0$ and $C_X\cap(W^{ss}(\rho,X)\setminus\{\rho\})=\emptyset$. Suppose $\wt{\cN}_{\cF(C_X)}$ admits a dominated splitting $\cN_1\oplus \cN_2$ such that $\dim \cN_1=i_0-1$.
Then there is a neighborhood $U$ of $C_X$ and a neighborhood $\cU$ of $X$ such that for any $Y\in\cU$ and any chain recurrence class $C_{Y}\subset U$ of $Y$, the class admits a dominated splitting $\wt{\cN}_{\cF(C_Y)}=\cN^Y_1\oplus\cN^Y_2$, such that $\cF(C_Y)\oplus \cN^Y_2$ projects to a continuous bundle $E^Y\subset T_{C_Y}M$ with $\dim E^Y=\dim M-i_0+1$.
\end{Proposition}

\begin{proof}	
	By robustness of dominated splitting, there exist a neighborhood $\cB$ of $\cF(C_X)$ and a neighborhood $\cU$ of $X$, such that for any $Y\in\cU$ and for any invariant set $\Delta\subset \cB$ of the extended tangent flow of $Y$, the normal bundle $\wt{\cN}(\Delta)$ admits a dominated splitting $\cN^Y_1\oplus \cN^Y_2$ with respect to the extend linear Poincar\'e flow $\tilde{\psi}_t^{Y}$, and it satisfies $\dim \cN^Y_1=i_0-1$.
	By Lemma \ref{lem.semi-cont}, there exists a neighborhood $U$ of $C_X$ such that for any $Y\in\cU$ (shrinking $\cU$ if necessary), for any chain recurrence class $C_{Y}\subset U$, one has $\cF(C_{Y})\subset\cB$. Hence there is a dominated splitting $\wt{\cN}_{\cF(C_{Y})}=\cN^Y_1\oplus \cN^Y_2$ with index $i_0-1$, i.e. $\dim \cN^Y_1=i_0-1$.
	
Since the singularities in $C_X$ are all Lorenz-like, by reducing $U$ and $\cU$, we can assume that singularities in $C_Y$ are also Lorenz-like. In particular, for any singularity $\rho\in C_Y$, there is a dominated splitting $T_{\rho}M=E^{ss}_{\rho}\oplus E^{cu}_{\rho}$ such that $\dim E^{ss}_{\rho}=i_0-1$. Moreover, by Lemma \ref{lem.loren-like-singularity} we can assume that for any $l\in \cF(C_Y)$, if $\beta(l)=\rho$ is a singularity, then $l\in E^{cu}_{\rho}$. As the dominated splitting $\wt{\cN}_{\cF(C_{Y})}=\cN^Y_1\oplus \cN^Y_2$ has index $i_0-1$, which equals to $\dim E^{ss}_{\rho}$, the uniqueness of domination implies that $\cN^Y_2(l)\subset E^{cu}_{\rho}$ and hence $\beta_*(l\oplus \cN^Y_2(l))=E^{cu}_{\rho}$.
Note that the subspace $E^{cu}_{\rho}$ is unique for all $l\in\cF(C_Y)\cap\beta^{-1}(\rho)$. For any $l\in\cF(C_Y)$ with $\beta(l)=x$ a regular point, we have $l=l^Y_x$. Thus, we can define $E^{cu}_x=\beta_*(l\oplus\cN_2(l))$. In this way, we define a bundle $E^{cu}$ over $C_Y$. By continuity of the map $\beta_*$ and the bundle $\cN^Y_2$, one can see that the bundle $E^{cu}$ is continuous at regular points. For any singularity $\rho\in C_Y$, let $x_n$ be a sequence in $C_Y$ such that $x_n\to\rho$ and $l_n=l^Y_{x_n}\to l$, we have
\[E^{cu}_{x_n}=\beta_*(l_n\oplus \cN_2^Y(l_n))\to \beta_*(l\oplus\cN_2^Y(l))=E^{cu}_{\rho}.\]
Hence $E^{cu}$ is also continuous at $\rho$. Therefore $\cF(C_Y)\oplus \cN^Y_2$ projects to the continuous bundle $E^{cu}$ with $\dim E^{cu}=\dim M-i_0+1$. 	
\end{proof}

\begin{proof}[Proof of Proposition \ref{prop.sect-exp}]
	Let $\wt{\cN}_{\cF(\Lambda)}=\cN_1\oplus \cN_2$ with $\dim \cN_2=1$ be the dominated splitting given by Proposition \ref{prop.away-tangencies}. Then $\cF(\Lambda)\oplus\cN_2$ projects to the two-dimensional continuous bundle $E^{cu}_{\Lambda}$, which is sectionally expanded. The neighborhoods $U\subset U_{\Lambda}$ of $\Lambda$ and $\cU$ of $X$ are given by Proposition \ref{prop.sectional-expanding}. By reducing these neighborhoods, we can assume that for any $Y\in\cU$ and any chain recurrence class $C_Y\subset U$, the splitting $\wt{\cN}_{\cF(C_Y)}=\cN^Y_1\oplus \cN^Y_2$ is close enough to $\wt{\cN}_{\cF(\Lambda)}=\cN_1\oplus \cN_2$. This implies that $\cF(C_Y)\oplus \cN^Y_2$ projects to a continuous bundle $E^Y$ that is close enough to $E^{cu}_{\Lambda}$, hence is also sectionally expanded. Moreover, as $E^Y$ contains $\beta_*(\cF(C_Y))$, it contains $Y(x)$ for all $x\in C_Y$.

Finally, if $C_Y=C(\si_Y,Y)$, then $C_Y$ varies upper semi-continuously with respect to $Y$. By Lemma \ref{lem.semi-cont}, $\cF(C_Y)$ also varies upper semi-continuously with respect to $Y$. Note that $E^Y=\beta_*(\cF(C_Y)\oplus \cN_2^Y)$. As the dominated splitting $\wt{\cN}_{F(C_Y)}=\cN^Y_1\oplus \cN^Y_2$ varies continuously with respect to $Y$ and the projection $\beta_*$ is continuous, one conclude that $E^Y$ varies upper semi-continuously with respect to $Y$.
\end{proof}

In the case $C_Y$ is non-singular, the sectional expanding property of the $E^Y$ bundle implies that the normal bundle $\cN^Y_2$ is expanded. This also holds for any non-singular compact invariant set in $U$.
\begin{Corollary}\label{cor.sect-exp}
  Let the neighborhoods $U$ and $\cU$ be given by Proposition \ref{prop.sect-exp}. For any $Y\in\cU$ and any compact $\phi^Y_t$-invariant set $\Lambda_Y\subset U$ such that $\Lambda_Y\cap\Sing(Y)=\emptyset$, the normal bundle $\cN_{\Lambda_Y}$ admits a dominated splitting $\cN_1\oplus \cN_2$ such that $\cN_2$ is one-dimensional and expanded.
\end{Corollary}

\section{Uniqueness of Lyapunov stable class: proof of Proposition \ref{prop.ls-class}}\label{sect.ls}

In this section we prove Proposition \ref{prop.ls-class}. In other words, we show that for $C^1$ perturbations of the vector field $X$, any Lyapunov stable chain recurrence class in a small neighborhood of $\Lambda$ contains the singularity $\sigma$ and the periodic orbit $Q$. We shall make use of the sectional expanding property of the $E^{cu}$ subbundle in the splitting $T_{\La}\Omega=E^{ss}\oplus E^{cu}\oplus \RR^I$. The difficulty lies in the fact that $E^{cu}$ is exceptional: there is no domination between the subbundles $E^{cu}$ and $\RR^I$. As a bypass, we will consider the dominated splitting $\wt{\cN}_{\cF(\Lambda)}=\cN_1\oplus \cN_2$ given by Proposition \ref{prop.away-tangencies}, where $\cN_1=\cN^{ss}\oplus \cN^I$.

The rough idea for the proof of Propostion \ref{prop.ls-class} goes as follows. We will first extend the dominated splitting $\wt{\cN}_{\cF(\Lambda)}=\cN_1\oplus \cN_2$ to a neighborhood $\cB_0$ of $\cF(\La)$ and define a cone around $\cN_2$ on the normal bundle (Section \ref{sect.cone-on-normal-bundle}), then use the cone on the normal bundle to construct an invariant cone on the cross-section $\Sigma$ for the first return map and show that vectors in the cone on $\Sigma$ are expanded by the first return map (Section \ref{sect.forward-inv} and Section \ref{sect.inv-cone}), and finally finish the proof of the proposition by an analysis of the first return map (Section \ref{sect.pf-prop-ls}).

\subsection{Dominated splitting and cone field on the normal bundle}\label{sect.cone-on-normal-bundle}
For simplicity, we assume that for the dominated splitting $\wt{\cN}_{\cF(\Lambda)}=\cN_1\oplus \cN_2$ the domination constant $T=1$. In other words,
\[\|\tilde{\psi}_t|_{\cN_1(l)}\|\cdot\|\tilde{\psi}_{-t}|_{\cN_2(l_t)}\|<\frac{1}{2},\quad\forall t\geq 1,\]
where $l_t=\hat{\Phi}_t(l)$.
We extend the splitting $\wt{\cN}_{\cF(\Lambda)}=\cN_1\oplus \cN_2$ to a neighborhood $\cB_0$ of $\cF(\Lambda)$ in a continuous way. Define on $\cB_0$ a cone field $\cC_{\alpha}$ on the normal bundle around the $\cN_2$ subbundle for any $\alpha>0$:
\[\cC_{\alpha}=\{v=v_1+v_2\in\cN_{\cB_0}: v_1\in\cN_1, v_2\in \cN_2, \|v_1\|\leq\alpha\|v_2\|\}.\]
We define also the cone\footnote{Technically, $\wt{\cC}_{\alpha}$ is not a cone, but a ``wedge''.} $\wt{\cC}_{\alpha}=\cB_0\oplus\cC_{\alpha}$. In other words, $v\in\wt{\cC}_{\alpha}$ if and only if there is a decomposition $v=v_1+v_2$ such that $v_1\in\cB_0$ and $v_2\in\cC_{\alpha}$.
By domination of the splitting $\cN_1\oplus\cN_2$ and continuity, we have the following lemma.
\begin{Lemma}\label{lem.cone-0}
	For any $\alpha>0$, there exist a neighborhood $\cU$ of $X$ and a neighborhood $\cB\subset\cB_0$ of $\cF(\Lambda)$ such that for any any $Y\in\cU$ and $l\in\cB$, if $\hat{\Phi}^Y_{[0,t]}(l)\subset \cB$ with $t\geq 1$, then
	\[\tilde{\psi}^Y_t(\cC_{\alpha}(l))\subset\cC_{\alpha/2}(\hat{\Phi}^Y_t(l)) \quad \text{and}\quad \tilde{\Phi}^Y_t(\wt{\cC}_{\alpha}(l))\subset\wt{\cC}_{\alpha/2}(\hat{\Phi}^Y_t(l)).\]
\end{Lemma}

Note that the vector bundle $\cF(\Lambda)\oplus \cN_2$ is invariant for the extended tangent flow $\tilde{\Phi}_t$.
By Remark \ref{rmk.exceptional-bundle}, $\beta_*(\cF(\Lambda)\oplus \cN_2)=E^{cu}_{\Lambda}$, which is sectionally expanded by the tangent flow.
Following from the continuity of the map $\beta_*$ and by reducing $\cB_0$, we can fix $\alpha_0>0$ small and a neighborhood $\cU_0$ of $X$ such that for any $Y\in\cU_0$ and $l\in\cB_0$, if $\hat{\Phi}^Y_{[0,t]}(l)\subset \cB_0$ with $t\geq 1$, then for any 2-dimensional subspace $\cP\subset\wt{\cC}_{\alpha_0}(l)$, the projection $\beta_*(\cP)$ is close enough to $E^{cu}_{\Lambda}$ and hence
	\begin{equation}\label{eq.sectional-expanding}
		\left|\det\left(\Phi^Y_t|_{\beta_*(\cP)}\right)\right|>e^{\gamma t/2},
	\end{equation}
where $\gamma>0$ is given as in \eqref{eq.sect-exp}.

We assume that the results of Lemma \ref{lem.cone-0} hold for the fixed $\alpha=\alpha_0$ with $\cU=\cU_0$ and $\cB=\cB_0$. Moreover, by Lemma \ref{lem.loren-like-singularity} we assume that
\begin{equation}\label{eq.lorenz-like}
W^{ss}(\sigma_Y,Y)\cap C(\sigma_Y,Y)=\{\sigma_Y\}, \quad \text{for any }\ Y\in\cU_0,
\end{equation}
where $W^{ss}(\sigma_Y,Y)$ is the strong stable manifold of $\sigma_Y$ tangent to $E^{ss}_{\sigma}\oplus \RR^I$.

\subsection{A forward invariance of \texorpdfstring{$\cB_0$}{B0}}\label{sect.forward-inv}

Recall that there is a cross-section $\Sigma=\Sigma_0\times I$ of the flow $\phi^X_t$. One can see that $\Sigma$ remains a cross-section for vector fields close to $X$. In this section we prove the following lemma.
\begin{Lemma}\label{lem.forward-inv}
  There exist a neighborhood $U_1$ of $\Lambda$ and a $C^1$ neighborhood $\cU_1$ of $X$, such that for any $p\in\Sigma\cap\overline{U_1}$, for any $Y\in\cU_1$, and $t\geq 0$, it holds $l^Y_{p_t}\in\cB_0$, where $p_t=\phi^Y_t(p)$.
\end{Lemma}

For the proof of Lemma \ref{lem.forward-inv} we need consider orbit segments arbitrarily close to the singularity.
Note that the tangent space at the singularity $\sigma$ admits a dominated splitting $T_{\sigma}\Omega=F^{ss}_{\sigma}\oplus E^{cu}_{\sigma}$, where $F^{ss}_{\sigma}=E^{ss}_{\sigma}\oplus \RR^I$; and by construction, it holds $\cF(\Lambda)\cap T_{\sigma}\Omega=E^{cu}_{\sigma}$.
One defines on $T_{\sigma}\Omega$ a $cu$-cone $C^{cu}_{\alpha}$ around the $E^{cu}_{\sigma}$ bundle. By domination, the cone $C^{cu}_{\alpha}$ is forward contracting. We extend the cone continuously to a ball around $\sigma$, which is denoted by $W_0$.
Let $K$ be the intersection of $\bigcup_{t\geq 0}\phi^X_t(L)$ with $\partial W_0$, which is a compact subset of $\partial W_0$.
We assume that the ball $W_0$ is small such that $W_0\cap \Sigma=\emptyset$ and since the singularity is Lorenz-like, there exists $\alpha>0$ such that for any $p\in K$, it holds $l^X_p\subset C^{cu}_{\alpha}(p)$. By reducing $\cU_0$ if necessary, there exists a neighborhood $N_K\subset \partial W_0$ of $K$ such that $l^Y_p\subset C^{cu}_{\alpha}(p)$ for any $p\in N_K$ and $Y\in\cU_0$.
Moreover, we assume that the cone $C^{cu}_{\alpha}$ is forward contracting under the tangent flow $\Phi^Y_t$ for any $Y\in\cU_0$.

\begin{Lemma}\label{lem.nbhd-of-sigma}
	There exists a neighborhood $W_1\subset W_0$ of $\si$ and a neighborhood $\cU_1\subset \cU_0$ of $X$ such that
for any $Y\in\cU_1$, any orbit segment from a point $p\in\Sigma$ to a point $q\in \overline{W_1}$ crosses $N_K$, and moreover, $l^Y_q\in \cB_0$.
\end{Lemma}
\begin{proof}
Since $\Sigma$ can be a cross-section for any vector field $C^1$ close to $X$ and its image under the first return map remains in the interior of $\Sigma$, we need only consider orbit segments from $p\in\Sigma$ to $q\in\overline{W_1}$ without returns to $\Sigma$.
Since $N_K$ is a neighborhood of $K$, there exists a neighborhood $W_1\subset W_0$ of $\si$ such that any orbit segment $S$ of $X$ from a point $p\in\Sigma$ to a point $q\in \overline{W_1}$ (without returns) crosses $N_K$. Moreover, let $r$ be the intersection of $S$ with $N_K$, then the orbit segment from $r$ to $q$ is contained in $\overline{W_0}$. See Figure \ref{fig.local-dynamics}. Shrinking $W_1$ if necessary, there exists a neighborhood $\cU_1\subset \cU_0$ of $X$ such that the same property holds for every $Y\in\cU_1$.
The following argument shows that $l^Y_q\in\cB_0$ if we choose $W_1$ small enough.

\begin{figure}[htbp]
\centering
  \includegraphics[width=.25\textwidth]{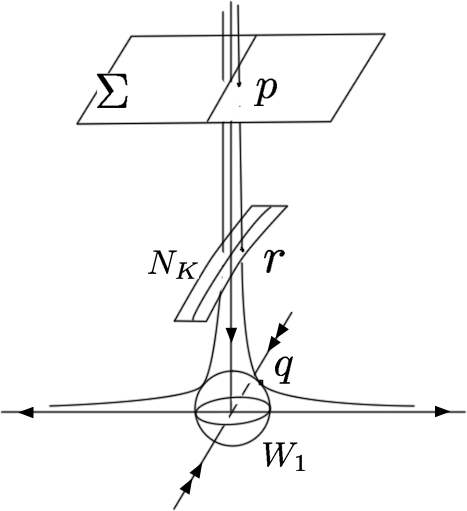}
  \caption{Local analysis around the singularity. Note \\ that the strong stable direction is 2-dimensional.}\label{fig.local-dynamics}
\end{figure}

Since $E^{cu}_{\sigma}\subset \cF(\Lambda)$ and $\cB_0$ is a neighborhood of $\cF(\La)$, we can assume that $W_1$ is small enough and there exists a constant $\alpha'>0$ such that for any $q\in \overline{W_1}$ and nonzero vector $v\in T_q\Omega$, if $v\in C^{cu}_{\alpha'}(q)$, then $\langle v\rangle \in \cB_0$; and
if the backward orbit of $q$ hits $N_K$ at a point $r$ before leaving $W_0$, then the time $t$ from $r$ to $q$ is large enough so that $\Phi^Y_t(C^{cu}_{\alpha}(r))\subset C^{cu}_{\alpha'}(q)$, by the contracting property of the cone. In particular, since $l^Y_r\subset C^{cu}_{\alpha}(r)$, we have $l^Y_q=\hat{\Phi}^Y_t(l^Y_r)\subset C^{cu}_{\alpha'}(q)$. It follows that $l^Y_q\in\cB_0$.
\end{proof}

The next result considers orbit segments away from the singularity.
\begin{Lemma}\label{lem.cone-1}
Reducing $\cU_1$ if necessary, there exists
an attracting neighborhood $U_1$ of $\Lambda$, such that for any $Y\in\cU_1$, one has $\phi^Y_1(\overline{U_1})\subset U_1$. Moreover, for any $p\in \overline{U_1}\setminus W_1$, it holds that $l^{Y}_{p_t}\in \cB_0$ for all $t\in[0,1]$, where $p_t=\phi^Y_t(p)$, and $\tilde{\psi}^{Y}_1(\cC_{\alpha_0}(l^{Y}_p))\subset \cC_{\alpha_0/2}(l^{Y}_{p_1})$.
\end{Lemma}
\begin{proof}
	Reducing $\cU_1$ if necessary, we can take a neighborhood $W'_1\subset W_1$ of $\sigma$ such that for any $p\in W'_1$ and any $Y\in\cU_1$, we have $\phi^Y_t(p)\in W_1$ for all $t\in[-1,1]$. Since $\sigma$ is the only singularity in $\Lambda$, we have 	\begin{equation}\label{eq.non-singular}
\cF(\Lambda)\setminus\beta^{-1}(W'_1)=\{l^X_p: p\in \Lambda\setminus W'_1\}.
\end{equation}
	By construction of the vector field $X$, there exist arbitrarily small attracting neighborhoods of $\Lambda$. Following from \eqref{eq.non-singular}, we can take an attracting neighborhood $U_1$ of $\Lambda$ and reduce the neighborhood $\cU_1$ of $X$ such that $\phi^Y_1(\overline{U_1})\subset U_1$ and  $l^Y_p\in\cB_0$ for any $Y\in\cU_1$ and for any $p\in \overline{U_1}\setminus W'_1$. This implies that for all $p\in\overline{U_1}\setminus W_1$ and $t\in[0,1]$, we have $l^{Y}_{p_t}\in \cB_0$, and hence by Lemma \ref{lem.cone-0}, $\tilde{\psi}^{Y}_1(\cC_{\alpha_0}(l^{Y}_p))\subset \cC_{\alpha_0/2}(l^{Y}_{p_1})$.
\end{proof}

We can now finish the proof of Lemma \ref{lem.forward-inv}.

\begin{proof}[Proof of Lemma \ref{lem.forward-inv}]
  Let the neighborhoods $\cU_1$ of $X$ and $U_1$ of $\Lambda$ be given by Lemma \ref{lem.nbhd-of-sigma} and Lemma \ref{lem.cone-1}.
  Let $p$ be any point in $\Sigma\cap \overline{U_1}$.
  By Lemma \ref{lem.cone-1}, $l^Y_{p_t}\in \cB_0$ for $t\in [0,1]$. If $p_1\in W_1$, let $t_1>1$ be such that the orbit segment $\phi^Y_{[1,t_1]}(p)\subset \overline{W_1}$ and $p_{t_1}\in\partial W_1$.
Then it follows from Lemma \ref{lem.nbhd-of-sigma} that $l^Y_{p_t}\in \cB_0$ for $t\in [1,t_1]$. If $p_1\notin W_1$, it follows from Lemma \ref{lem.cone-1} again that $l^Y_{p_t}\in \cB_0$ for any $t\in[1,2]$. Inductively, one concludes that $l^Y_{p_t}\in\cB_0$ for all $t\geq 0$.
\end{proof}

\subsection{Invariant cone on the cross-section and expanding property}\label{sect.inv-cone}
Recall that $R:\Sigma\setminus L\to \Sigma$ is the first return map, where $L$ is the intersection of the local stable manifold of $\sigma$ with $\Sigma$ (Section \ref{sect.skew-product}).
For any $Y\in\cU_1$, let $L_Y=W^s_{loc}(\sigma_Y,Y)\cap\Sigma$ and $R^Y:\Sigma\setminus L_Y\to \Sigma$ be the first return map. Let $TR^Y$ be the tangent map of $R^Y$. For simplicity, let us denote
\[\Sigma^*=\Sigma\cap \overline{U_1}.\]
We define in this section a cone on $\Sigma^*$ such that it is invariant under $TR^Y$ and moreover, vectors in the cone will be shown to be expanded by $TR^Y$. This allows us to consider curves tangent to the cone on $\Sigma^*$ and show that the length of such curves increase under the iteration of the first return map.

In the following, for any $0<\alpha\leq\alpha_0$ and $p\in \overline{U_1}\setminus \{\sigma_Y\}$ such that $l^Y_p\in\cB_0$ we denote $\cC^Y_{\alpha}(p)=\beta_*(\wt{\cC}_{\alpha}(l^Y_p))$, which is a cone on the tangent space $T_p\Omega$. We then define a cone field $\cD^Y_{\alpha}$ on $\Sigma^*$ by letting $\cD^Y_{\alpha}(p)=\cC^Y_{\alpha}(p)\cap T_p\Sigma$.
By shrinking $U_1$ and $\cU_1$, the following result holds.
\begin{Lemma}\label{lem.expanding}
  There exist $0<\alpha_1<\alpha_0$ and $\lambda>1$ such that for any $Y\in\cU_1$, any $p\in\Sigma^*\setminus L_Y$ and any $v\in \cD^X_{\alpha_1}(p)$, one has $TR^Y(v)\in \cD^X_{\alpha_1}(p')$, where $p'=R^Y(p)$. Moreover, $\|TR^Y(v)\|\geq\lambda\|v\|$.
\end{Lemma}
\begin{proof}
   Note that restricted to $\Sigma^*$, the cone $\cC^Y_{\alpha}(p)$ and hence $\cD^Y_{\alpha}(p)$ are uniformly continuous with respect to the vector field $Y$ in the $C^1$ topology. Let us fix $\alpha_1\in(\alpha_0/2,\alpha_0)$. Then by reducing $\cU_1$, we can assume that for any $Y\in\cU_1$, it holds
   \[\cD^Y_{\alpha_0/2}(p)\subset \cD^X_{\alpha_1}(p)\subset \cD^Y_{\alpha_0}(p).\]
   Also, from (C1) in Section \ref{sect.skew-product} we can assume that $t_p>2$ for any $p\in \Sigma^*\setminus L_Y$. Now, for any $v\in \cD^X_{\alpha_1}(p)\subset \cD^Y_{\alpha_0}(p)$, we have $(l^Y_p,v)\in \wt{\cC}^Y_{\alpha_0}(l^Y_p)$. From Lemma \ref{lem.forward-inv} it follows that $l^Y_{p_t}\in\cB_0$ for all $t\geq 0$. Thus, by Lemma \ref{lem.cone-0}, we have
   $\wt{\Phi}^Y_{t_p}(l^Y_p,v)\in\wt{\cC}_{\alpha_0/2}(l^Y_{p'})$, where $p'=R^Y(p)$. Equivalently, we have $\Phi^Y_{t_p}(v)\in \cC^Y_{\alpha_0/2}(p')$. Projecting $\Phi^Y_{t_p}(v)$ to $T_{p'}\Sigma$ along $l^Y_{p'}$, we obtain that
   \[TR^Y(v)\in \cD^Y_{\alpha_0/2}(p')\subset \cD^X_{\alpha_1}(p').\]
   This establishes the invariance of the cone $\cD^X_{\alpha_1}$.

   To see that $TR^Y$ is expanding on $\cD^X_{\alpha_1}$, we shall refer to the sectional expanding property \eqref{eq.sectional-expanding}. Recall that in the construction of the example, we have assumed the vector field $X^0$ to be orthogonal to the cross-section $\Sigma_0$ and $\|X^0(x)\|=1$ for $x\in \Sigma_0$, see (P2) in Section \ref{sect.lorenz}. Then it follows from \eqref{eq.vector-field} that the vector field $X$ is orthogonal to the cross-section $\Sigma$ at points $p\in \Sigma_0\times\{0\}$ and $\|X(p)\|=1$. In particular, this holds for $p\in \Lambda\cap\Sigma$. By shrinking $U_1$ and $\cU_1$ if necessary, we can assume that for any $Y\in\cU_1$ and any $p\in \Sigma^*=\Sigma\cap\overline{U_1}$, $Y(p)$ is almost orthogonal to the cross-section and $\|Y(p)\|$ is close to 1. Then the sectional expanding property \eqref{eq.sectional-expanding} allows us to obtain a constant $\lambda>1$, independent of $Y$, such that $\|TR^Y(v)\|\geq \lambda\|v\|$.
\end{proof}

The previous result shows that the cone $\cD^X_{\alpha_1}$ on $\Sigma^*$ is invariant for $TR^Y$, for any $Y\in\cU_1$. Moreover, $TR^Y$ is expanding on $\cD^X_{\alpha_1}$.
Note the first return time $t_p$ of a point $p\in\Sigma^*\setminus L_Y$ can be arbitrarily large if it is close enough to $L_Y$. In this case, the sectional expanding property \eqref{eq.sectional-expanding} will guarantee a large expansion rate for $TR^Y$.

\begin{Lemma}\label{lem.expanding-2}
There exists a neighborhood $N_L$ of $L$ in $\Sigma$ such that for any $Y\in\cU_1$ (reducing $\cU_1$ if necessary), one has $L_Y\subset N_L$ and for any $p\in (N_L\setminus L_Y)\cap \Sigma^*$, 	
\begin{equation}\label{eq.expanding-2}
  \|TR^Y(v)\|>3\|v\|,\quad \forall v\in \cD^X_{\alpha_1}(p)\setminus\{0\}.
\end{equation}
\end{Lemma}

\begin{proof}
    By Lemma \ref{lem.forward-inv} and the sectional expanding property \eqref{eq.sectional-expanding}, there exists $t_0>0$ such that for any $Y\in\cU_1$ and any $p\in \Sigma^*\setminus L_Y$, if the first return time $t_p>t_0$, then equation \eqref{eq.expanding-2} holds.
    By continuity of local stable manifold, for any neighborhood $N_L$ of $L$ in $\Sigma$ there exists a $C^1$ neighborhood $\cU$ of $X$ such that for any $Y\in\cU$, the intersection $L_Y=W^s_{loc}(\sigma_Y,Y)\cap \Sigma$ is contained in $N_L$. Therefore, by reducing $\cU_1$, we can assume that this property holds for $Y\in\cU_1$ and $N_L$ is small enough such that for any $p\in (N_L\setminus L_Y)\cap \Sigma^*$, the first return time $t_p$ is larger than $t_0$. Hence equation \eqref{eq.expanding-2} holds.
\end{proof}

\begin{Definition}[$cu$-curve]
  A $C^1$ curve $J$ on $\Sigma$ is called a {\em $cu$-curve} if it is contained in $\Sigma^*$ and is tangent to the cone $\cD^X_{\alpha_1}$.
\end{Definition}

\begin{Lemma}\label{lem.cu-curve0}
  For any $Y\in\cU_1$ and any $cu$-curve $J$ such that $J\cap L_Y=\emptyset$, the image $R^Y(J)$ remains a $cu$-curve.
\end{Lemma}
\begin{proof}
  Since $J\cap L_Y=\emptyset$, for any $p\in J$, there is a neighborhood $J_p\subset J$ of $p$ and a constant $\delta>0$ such that the set $\Gamma=\bigcup_{|t-t_p|\leq \delta}\phi^Y_t(J_p)$ is a 2-dimensional submanifold tangent to the cone $\cC^Y_{\alpha_0}$, transverse to $\Sigma$. The intersection $\Gamma\cap \Sigma$ is contained in $\Sigma^*$ and by Lemma \ref{lem.expanding}, is a $cu$-curve. This implies that $R^Y(J)$ is a $cu$-curve.
\end{proof}

Assuming $\cU_1$ small enough, the following result holds.

\begin{Lemma}\label{lem.cu-curve2}
	There exists $\vep_0>0$ such that for any $Y\in\cU_1$ and any $cu$-curve $J$, there exists $k>0$ such that $\bigcup_{n=0}^{k-1} (R^Y)^n(J)$ contains a $cu$-curve with length larger than $\vep_0$.
\end{Lemma}

\begin{proof}
Assume that the $C^1$ neighborhood $\cU_1$ of $X$ is small enough such that $L_Y$ is $C^1$ close to $L$ and transverse to $\cD^X_{\alpha_1}$ with uniform angle. Moreover, $N_L$ contains an $\varepsilon$-neighborhood of $L_Y$ in $\Sigma$, where $\vep>0$ is uniform for $Y\in\cU_1$. Then there exists $\vep_0\in(0,\vep)$ such that any $cu$-curve with length smaller than $\varepsilon_0$ intersects $L_Y$ at most once, and moreover, when it does intersect $L^Y$ then it is contained in $N_L$.

Let $J$ be any $cu$-curve that does not intersect $L_Y$. By Lemma \ref{lem.cu-curve0}, $R^Y(J)$ remains a $cu$-curve. Moreover, Lemma \ref{lem.expanding} shows that $\len(R^Y(J))>\lambda\len(J)$, where $\len(\cdot)$ denotes the length of a $C^1$ curve. By iteration, the length of $(R^Y)^n(J)$ keeps growing if it does not intersect $L_Y$. When $(R^Y)^n(J)$ does intersect $L_Y$, then either $(R^Y)^n(J)$ has length larger than $\varepsilon_0$, or it is contained in $N_L$ and is cut by $L_Y$ into two pieces. In the latter case, let $J'$ be the longer piece. Then Lemma \ref{lem.expanding-2} implies that $R^Y(J')$ has length larger than $3/2$ times the length of $(R^Y)^n(J)$. The conclusion of the lemma holds by induction.
\end{proof}

\subsection{Proof of Proposition \ref{prop.ls-class}}\label{sect.pf-prop-ls}

We now fix the neighborhoods $U_1$ and $\cU_1$ small enough so that all the results on Section \ref{sect.forward-inv} and Section \ref{sect.inv-cone} hold.

Recall that there is a periodic orbit $Q$ in $\Lambda$ with stable index 2 such that its stable manifold is dense in $U_{\Lambda}$, see (C2) in Section \ref{sect.skew-product}. In particular, the stable manifold of $Q$ intersects $\Sigma$ along a dense family of 2-dimensional $C^1$ disks of the form $\cF^s(x)\times I$. As is easy to see that the foliation $\cF^s\times I$ is transverse to the bundle $E^{cu}$, we can assume that $\alpha_0$ is small such that $\cF^s\times I$ is transverse to the cone $\cD^X_{\alpha_0}$. This implies, in particular, that any $cu$-curve is transverse to the foliation $\cF^s\times I$ and the angle between them is uniformly bounded below.
By continuity of local stable manifold, there exists a $C^1$ neighborhood $\cU_2$ of $X$ such that for any $Y\in\cU_2$, any $cu$-curve with length $\vep_0$ intersects the stable manifold $W^s(Q_Y,Y)$, where $\vep_0>0$ is given by Lemma \ref{lem.cu-curve2} and $Q_Y$ is the continuation of $Q$.

\begin{Lemma}\label{lem.unique-ls-class}
  For any $Y\in\cU_1\cap\cU_2$, for any $\phi^Y_t$-invariant compact set $\Gamma\subset U_1$ such that $\Gamma\cap\Sing(Y)=\emptyset$, it holds $W^u(\Gamma,Y)\cap W^s(Q_Y,Y)\neq\emptyset$ and $W^u(\Gamma,Y)\cap W^s(\sigma_Y,Y)\neq\emptyset$.
\end{Lemma}

\begin{proof}
  By Corollary \ref{cor.sect-exp}, there is a dominated splitting $\cN_{\Gamma}=\cN_1\oplus \cN_2$ such that $\cN_2$ is one-dimensional and expanding. This implies that $\Gamma$ has an unstable manifold tangent to $\langle Y\rangle\oplus\cN_2$. By Lemma \ref{lem.cone-0}, the cone $\cD^X_{\alpha_1}$ is invariant under $TR^Y$, which implies that $(l^Y_p\oplus \cN_2(p))\cap T_p\Sigma\subset \cD^X_{\alpha_1}(p)$ for any $p\in \Gamma\cap \Sigma$. Then the intersection of the unstable manifold of $\orb(p)$ with $\Sigma$ contains a $cu$-curve $J$. It follows from Lemma \ref{lem.cu-curve2} that the length of $J$ can be assumed to be larger than $\vep_0$. Since $Y\in\cU_2$, $J\cap W^s(Q_Y,Y)\neq\emptyset$. Hence $W^u(\Gamma,Y)\cap W^s(Q_Y,Y)\neq\emptyset$.

  For the proof of $W^u(\Gamma,Y)\cap W^s(\sigma_Y,Y)\neq\emptyset$, we need only show that $W^u(\Gamma,Y)\cap L_Y\neq\emptyset$. Assume that $J$ does not intersect $L_Y$, then by invariance of the cone $\cD^X_{\alpha_1}$ and Lemma \ref{lem.expanding}, the iterate $R^Y(J)$ remains a $cu$-curve with length larger than $\lambda \len(J)$. By iteration, the length of $(R^Y)^n(J)$ keeps growing if the first $n-1$ iterates does not intersects with $L_Y$. As the cone $\cD^X_{\alpha_1}$ is transverse to $\cF^s\times I$, there exists a finite upper bound for the length of $cu$-curves. Hence there exists $n>0$ such that $(R^Y)^n(J)\cap L_Y\neq\emptyset$. Therefore, $W^u(\Gamma,Y)\cap L_Y\neq\emptyset$.
\end{proof}

\begin{Corollary}\label{cor.unique-ls-class}
  For any $Y\in \cU_1\cap\cU_2$, the chain recurrence class $C(\sigma_Y,Y)$ is the only Lyapunov stable one contained in $U_1$.
\end{Corollary}

\begin{proof}
  Let $C_Y$ be any Lyapunov stable chain recurrence class of $Y$ contained in $U_1$. Suppose that $C_Y$ does not contain the singularity $\sigma_Y$. Then $C_Y$ is non-singular. By Lemma \ref{lem.unique-ls-class}, we have $W^u(C_Y,Y)\cap W^s(\sigma_Y,Y)\neq\emptyset$.
  By Lyapunov stability of the chain recurrence class $C_Y$, one would obtain that $\sigma_Y\in \Cl(W^u(C_Y,Y))\subset C_Y$, a contradiction. Thus, $C_Y$ contains $\sigma_Y$. It follows that $C(\sigma_Y,Y)$ is the only Lyapunov stable chain recurrence class in $U_1$.
\end{proof}

Let $Y\in \cU_1\cap \cU_2$, we need to show that $Q_Y$ is contained in the Lyapunov stable chain recurrence class $C(\sigma_Y,Y)$.
Let us recall the following result in \cite{GY}.
\begin{Lemma}[{\cite[Proposition 4.9]{GY}}]\label{lem.mixingdominated}
	Let $\Delta$ be a compact invariant set of $Y\in \xX^1(M)$ verifying the following properties:
	\begin{itemize}
		\item $\Delta\setminus{\rm Sing}(Y)$ admits a dominated splitting
			$\cN_{\Delta}=G^{cs}\oplus G^{cu}$ in the normal bundle w.r.t. the linear Poincar\'e flow $\psi_t$, with index $i$.
		\item Every singularity $\rho\in\Delta$ is hyperbolic and $\ind(\rho)>i$. Moreover, $T_\rho M$ admits a partially hyperbolic splitting $T_{\rho}M=F^{ss}\oplus F^{cu}$ with respect to the tangent flow, where $\dim F^{ss}=i$ and for the corresponding strong stable manifolds $W^{ss}(\rho)$, one has $W^{ss}(\rho)\cap \Delta=\{\rho\}$.
		\item For every $x\in\Delta$, one has $\omega(x)\cap {\rm Sing}(Y)\neq\emptyset$.
	\end{itemize}
	Then either $\Delta$ admits a partially hyperbolic splitting $T_\Delta M=E^{ss}\oplus F$ with respect to the tangent flow $\Phi^Y_t$, where $\dim E^{ss}=i$, or $\Delta$ intersects the homoclinic class $H(\gamma)$ of a hyperbolic periodic orbit $\gamma$ of index $i$.
\end{Lemma}

Since $Y\in\cU_1$, the singularity $\sigma_Y$ is Lorenz-like and it follows from Proposition \ref{prop.sectional-expanding} that $\cN_{C(\sigma_Y,Y)}$ admits a dominated splitting $\cN_1\oplus\cN_2$ with index $2$. Moreover, since $\cU_1\subset \cU_0$ and by \eqref{eq.lorenz-like}, we have $W^{ss}(\sigma_Y,Y)\cap C(\sigma_Y,Y)=\{\sigma_Y\}$. Let us consider the following two cases:
\begin{enumerate}[label=(\arabic*)]
  \item either $\omega(x)\cap\Sing(Y)\neq\emptyset$, for every $x\in C(\sigma_Y,Y)$;
  \item or there exists $x\in C(\sigma_Y,Y)$ such that $\omega(x)\cap\Sing(Y)=\emptyset$, i.e. $\sigma_Y\notin\omega(x)$.
\end{enumerate}
If Case (1) happens, then all the conditions in Lemma \ref{lem.mixingdominated} are satisfied for $C(\sigma_Y,Y)$ and since $C(\sigma_Y,Y)$ can not intersect a homoclinic class, it follows that $C(\sigma_Y,Y)$ admits a partially hyperbolic splitting $T_{C(\sigma_Y,Y)}M=F^{ss}\oplus F^{cu}$ with respect to the tangent flow. One can see that the bundle $F^{cu}$ is exactly the sectionally expanding subbundle $E^Y$ given by Proposition \ref{prop.sect-exp}. Therefore, $C(\sigma_Y,Y)$ is singular hyperbolic, and by Corollary \ref{cor.unique-ls-class} it is also Lyapunov stable. By the result \cite[Corollary D]{PYY21}, the class $C(\sigma_Y,Y)$ contains a periodic orbit, which is a contradiction to the assumption that $\omega(x)\cap\Sing(Y)\neq\emptyset$ for all $x\in C(\sigma_Y,Y)$.
So we are left with Case (2), in which $C(\sigma_Y,Y)$ contains a non-singular compact invariant subset $\Gamma$. Then it follows from Lemma \ref{lem.unique-ls-class} that $W^u(\Gamma,Y)\cap W^s(Q_Y,Y)\neq\emptyset$. Hence $W^u(C(\sigma_Y,Y),Y)\cap W^s(Q_Y,Y)\neq\emptyset$. Applying Lemma \ref{lem.unique-ls-class} to $\Gamma=Q_Y$, we also obtain that $W^u(Q_Y,Y)\cap W^s(\sigma_Y,Y)\neq\emptyset$. Thus $Q_Y\subset C(\sigma_Y,Y)$.

This ends the proof of Proposition \ref{prop.ls-class}.

\noindent Ming Li, School of Mathematical Sciences, Nankai University, Tianjing, China\\
E-mail: limingmath@nankai.edu.cn
\vspace{0.5cm}

\noindent Fan Yang, Department of Mathematics, Michigan State University, MI, USA \\
E-mail: yangfa31@msu.edu
\vspace{0.5cm}

\noindent Jiagang Yang, Departamento de Geometria, Instituto de Matem\'{a}tica e
Estat\'{i}stica, Universidade Federal Fluminense, Niter\'{o}i, Brazil\\
E-mail: yangjg@impa.br
\vspace{0.5cm}

\noindent Rusong Zheng, Joint Research Center on Computational Mathematics and Control, Shenzhen MSU-BIT University, Shenzhen, China\\
E-mail: zhengrs@smbu.edu.cn

\end{document}